\documentclass[12pt]{amsart}  

\usepackage{amscd} 
\usepackage{mathrsfs}
\usepackage{xcolor}
\usepackage{ amssymb}

\newtheorem{theorem}{Theorem}[section]
\newtheorem{lemma}[theorem]{Lemma}
\newtheorem{corollary}[theorem]{Corollary} 
\theoremstyle{definition}
\newtheorem{definition}[theorem]{Definition}

\newtheorem{remark}[theorem]{Remark} 
\numberwithin{equation}{section}
\newtheorem{notation}[theorem]{Notation}

\def\z*{z^{*}}

\def\B{\mathscr B}
\def\M{\mathscr M}

\def\uno{\mathsf 1}
\def\H{\mathsf H}
\def\K{\mathsf K}
\def\X{\mathsf X}

\def\Y{\mathsf Y}

\def\C{\mathscr C}

\def\dom{\text{\rm dom}}
\def\ran{\text{\rm ran}}
\def\supp{\text{\rm supp}}

\def\RE{\mathbb R}
\def\CO{{\mathbb C}}

\def\SL{S\! L}
\def\DL{D\! L}
\def\U{U}

\def\ph*{\phi_\star}

\def\be{\begin{equation}}
\def\ee{\end{equation}}

\def\-{{\rm in}}
\def\+{{\rm ex}}
\def\sgn{\text{sgn}}
\begin{document}

\title[Inverse Scattering for the Laplace operator]
{Inverse Scattering for the Laplace operator with boundary conditions on Lipschitz surfaces }
\author{Andrea Mantile}
\author{Andrea Posilicano}

\address{Laboratoire de Math\'{e}matiques, Universit\'{e} de Reims -
FR3399 CNRS, Moulin de la Housse BP 1039, 51687 Reims, France}
\address{DiSAT, Sezione di Matematica, Universit\`a dell'Insubria, via Valleggio 11, I-22100
Como, Italy}
\email{andrea.mantile@univ-reims.fr}
\email{andrea.posilicano@uninsubria.it}

\begin{abstract} We provide a general scheme, in the combined frameworks of Mathematical Scattering Theory and Factorization Method, for inverse scattering for the couple of self-adjoint operators $(\widetilde\Delta,\Delta)$, where $\Delta$ is the free Laplacian in $L^{2}(\RE^{3})$ and $\widetilde\Delta$ is one of its singular perturbations, i.e., such that the set $\{u\in H^{2}(\RE^{3})\cap \dom(\widetilde\Delta)\, :\, \Delta u=\widetilde\Delta u\}$ is dense. Typically $\widetilde\Delta$ corresponds to a self-adjoint realization of the Laplace operator with some kind of boundary conditions imposed on a null subset;  in particular our results apply to standard, either separating or semi-transparent, boundary conditions at $\Gamma=\partial\Omega$, where  $\Omega\subset\RE^{3}$ is a bounded Lipschitz domain. Similar results hold in the case  the boundary conditions are assigned only on $\Sigma\subset\Gamma$, a relatively open subset with a Lipschitz boundary. We show that either the obstacle $\Omega$ or the screen $\Sigma$ are determined by the knowledge of the Scattering Matrix, equivalently of the Far Field Operator, at a single frequency. 
\end{abstract}

\maketitle
\section{Introduction}
In the recent paper \cite{JMPA} (also see \cite{JST} for the case of smooth boundaries and \cite{BMN} for similar results in the case of smooth boundaries and under additional trace-class conditions) we obtained a representation formula for the scattering matrix $S^{\Lambda}_{\lambda}:L^{2}({\mathbb S}^{2})\to L^{2}({\mathbb S}^{2})$ relative to the scattering couple $(\Delta,\Delta_{\Lambda})$, where $\Delta$ is the self-adjoint free Laplacian in $L^{2}(\RE^{3})$ and $\Delta_{\Lambda}$ is a self-adjoint realization of the Laplacian with boundary conditions at $\Gamma$, the Lipschitz boundary of the bounded domain $\Omega\subset\RE^{3}$. Here $\Lambda:z\mapsto\Lambda_{z}$ is an operator-valued map which univocally defines $\Delta_{\Lambda}$ and fixes the boundary conditions realized by the corresponding operator (see Sections \ref{appli} and \ref{appli-screens} below for various explicit examples). Our representation formula gives $S_{\lambda}^{\Lambda}=\uno-2\pi iL_{\lambda}\Lambda_{\lambda}^{+}L^{*}_{\lambda}$, where $\Lambda_{\lambda}^{+}$ is the limit of $\Lambda_{\lambda+i\epsilon}$ as $\epsilon\downarrow 0$ (which, under suitable hypotheses,  exists  in operator norm through a Limiting Absorption Principle, see \cite{JMPA}), and $L_{\lambda}$ is defined in term of the trace (either Dirichlet or Neumann or both) at the boundary $\Gamma$ of the free waves with  wavenumber $|\lambda|^{1/2}$. Introducing the Far-Field operator $F^{\Lambda}_{\lambda}:=\frac1{2\pi i}(\uno-S^{\Lambda}_{\lambda})$ (see \cite[relation (1.31)]{KG}), one gets $F^{\Lambda}_{\lambda}=L_{\lambda}\Lambda_{\lambda}^{+}L^{*}_{\lambda}$; such a factorized form  suggests to study the inverse scattering problem (concerning the reconstruction of the shape of $\Omega$ by the knowledge of the scattering data at a fixed frequency) by means of Kirsch's Factorization Method (see \cite{KG} and references therein). 
Our result is the following (see Theorem \ref{cor}  for the complete statement): let $\Lambda^{+}_{\lambda}=(M^{+}_{\lambda})^{-1}$, where the bounded bijective operator $M^{+}_{\lambda}$ has the decomposition $M_{\lambda}^{+}=M_{\circ}+K_{\lambda}$, $M_{\circ}$ sign-definite and $K_{\lambda}$ compact; then  
$$
x\in\Omega\quad\iff\inf_{\substack{\psi\in L^{2}({\mathbb S}^{2})\\ \langle\psi,\phi^{x}_{\lambda}\rangle_{L^{2}({\mathbb S}^{2})}=1}}\left|\langle\psi,F^{\Lambda}_{\lambda}\psi\rangle_{L^{2}({\mathbb S}^{2})}\right|>0\iff \sum_{k=1}^{\infty}\frac{|\langle\phi^{x}_{\lambda},\psi^{\Lambda}_{\lambda,k}\rangle_{L^{2}({\mathbb S}^{2})}|^{2}}{|z^{\Lambda}_{\lambda,k}|}<+\infty\,,
$$
where $\phi^{x}_{\lambda}(\xi):=e^{i\,|\lambda|^{1/2}\xi\cdot x}$ and the sequences $\{z^{\Lambda}_{\lambda,k}\}_{1}^{\infty}\subset\CO\backslash\{0\}$ and $\{\psi^{\Lambda}_{\lambda,k}\}_{1}^{\infty}\subset L^{2}({\mathbb S}^{2})$ provide the spectral resolution of the compact normal operator $F^{\Lambda}_{\lambda}$. While such a result conforms to the standard ones (the inf-criterion and the $(F^{*}F)^{1/4}$-method) provided in \cite[Section 1.4]{KG}, its advantage is due to the fact that 
we use a factorization where all the informations regarding the boundary conditions are encoded in the operator $\Lambda_{\lambda}^{+}$, whereas $L_{\lambda}$, for which one needs to characterize the  range, is model-independent; this enhances the flexibility of our approach. Moreover, with a minimal effort (which in essence consists in compressing the operator $\Lambda_z$ onto subspaces of functions with supports contained in $\Sigma\subset\Gamma$) one gets similar results in the case the boundary conditions are imposed not on the whole $\Gamma$ but only on a relatively open subset $\Sigma$ with a Lipschitz boundary. In this case the result is of the same kind, only the family of testing functions changes (see Theorem \ref{screen} for the precise statement): let $\Sigma_{\circ}\subset\Gamma_{\!\circ}$, $\Gamma_{\!\circ}$ a Lipschitz boundary; then
\begin{align*}
\Sigma_{\circ}\subset\Sigma\iff\inf_{\substack{\psi\in L^{2}({\mathbb S}^{2})\\ \langle\psi,\phi^{\Sigma_{\circ}}_{\lambda}\rangle_{L^{2}({\mathbb S}^{2})}=1}}\left|\langle\psi,F^{\Lambda}_{\lambda}\psi\rangle_{L^{2}({\mathbb S}^{2})}\right|>0
\iff\sum_{k=1}^{\infty}\frac{|\langle\phi^{\Sigma_{\circ}}_{\lambda},\psi^{\Lambda}_{\lambda,k}\rangle_{L^{2}({\mathbb S}^{2})}|^{2}}{|z^{\Lambda}_{\lambda,k}|}<+\infty\,,
\end{align*}
where
$\phi^{\Sigma_{\circ}}_{\lambda}(\xi):=\int_{\Sigma_{\circ}}\phi_{\lambda}^{x}(\xi)\,d\sigma_{\Gamma_{\!\circ}}(x)$. \par
We provide several examples where Theorems \ref{cor} and \ref{screen} apply. In particular, we consider obstacles and screens reconstruction for the following boundary conditions:
\par\noindent $\bullet$
Dirichlet  $\gamma_{0}u=0$ (see Subsections \ref{Dir-ob} and \ref{Ds});
\par\noindent $\bullet$
Neumann $\gamma_{1}u=0$ (see Subsections \ref{Neu-ob} and \ref{Ns});
\par\noindent $\bullet$ 
semi-transparent  $$\begin{cases}
\alpha\gamma_{0}u=[\gamma_{1}]u\,,\\
[\gamma_{0}]u=0\,,\end{cases}$$ 
either $\alpha>0$ or $\alpha<0$ (see Subsections \ref{delta-ob} and \ref{delta-s}); 
\par\noindent $\bullet$
semi-transparent $$\begin{cases}
\gamma_{1}u=\theta[\gamma_{0}]u\,,\\
[\gamma_{1}]u=0\,,\end{cases}
$$ 
$\theta$ real-valued (see Subsections \ref{delta'-ob} and \ref{delta'-s});
\par\noindent $\bullet$
local of the kind 
\be\label{locB}\begin{cases}
\gamma_{0}u=b_{11}[\gamma_{0}]u+b_{12}[\gamma_{1}]u\,,\quad\\
\gamma_{1}u=b^{*}_{12}[\gamma_{0}]u+b_{22}[\gamma_{1}]u\,,
\end{cases}
\ee 
$b_{11}<0$, $b_{22}$ real-valued (see Subsections \ref{rob-ob} and \ref{rob-s}).
\par
A huge literature is devoted to obstacle reconstruction from scattering data; we just recall some papers where the Factorization Method is used in connection with the models here treated. Dirichlet and Neumann obstacles have been considered in \cite{Kirsch98} (see also \cite[Chap. 1]{KG}); Dirichlet screens have been studied firstly, in a 2-dimensional setting, in \cite{Kirsch00}.  Semi-transparent interface conditions appear, apart in quantum mechanical models (see, e.g., \cite{BEKS}, \cite{BLL} and references therein), in  connections with acoustic models with gradient singularities, see \cite{JDE2}. Conditions of the type $\alpha\gamma_{0}=[\gamma_{1}]u$ appear in \cite{KK} and \cite{BL13} in a non self-adjoint setting (i.e. when $\alpha$ is complex-valued): this compels the use of different data operators. 
An appropriate choice of the functions $b_{ij}$ in \eqref{locB} gives the classical Robin boundary conditions; the latter have been considered in \cite{Grinb06} (see also \cite[Chap. 2]{KG}) and \cite{BH13}. In these papers, as in the previous case, a non self-adjoint setting is used and different data operators enters in the reconstruction formulae.\par
In this paper, as regards scattering, we use a quantum mechanics point of view (see Section \ref{WaveOp}); however, as recalled in Section \ref{swe} below (see also \cite{W} for the case of Neumann boundary conditions), the scattering theory for Schr\"odinger-type equations is equivalent to the one for wave-type equations. Hence our reconstruction results apply to diffusions of both classical and quantum waves.\par
In order to simplify the exposition, our results are stated in dimension $d=3$; however they hold in any dimension $d\ge 2$. Finally, we presume that, by the same techniques, our approach can be extended to the case in which  the Laplace operator is replaced by a more general 2nd order elliptic differential operator.    
\vskip10pt\noindent
{\bf Acknowledgements.} The authors are indebted to Mourad Sini for the 
fruitful discussions which largely inspired this work.

\section{Notations and preliminaries.}
\subsection{Notations.}{\ }
\vskip5pt \noindent $\bullet$  $\|\cdot\|_{\X}$ denotes the norm on the complex Banach space $\X$; in case $\X$ is a Hilbert space, $\langle\cdot,\cdot\rangle_{\X}$ denotes the (conjugate-linear w.r.t. the first argument) scalar product.
\vskip5pt\noindent $\bullet$ $\langle\cdot,\cdot\rangle_{\X^{*},\X}$ denotes the duality (assumed to be conjugate-linear w.r.t. the first argument) between the dual couple $(\X^{*},\X)$.
\vskip5pt\noindent $\bullet$ $L^{*}:\dom(L^{*})\subseteq \Y^{*}\to \X^{*}$ denotes the dual of the densely defined linear operator $L:\dom(L)\subseteq \X\to \Y$; in a Hilbert spaces setting $L^{*}$ denotes the adjoint operator.
\vskip5pt\noindent $\bullet$ $\rho(A)$ and $\sigma(A)$ denote the resolvent set and the spectrum of the self-adjoint operator $A$; $\sigma_{\rm p}(A)$, $\sigma_{\rm pp}(A)$, $\sigma_{\rm ac}(A)$,  $\sigma_{\rm sc}(A)$, $\sigma_{\rm ess}(A)$, $\sigma_{\rm disc}(A)$, denote the point, pure point, absolutely continuous, singular continuous, essential and discrete spectra.
\vskip5pt\noindent $\bullet$ $\B(\X,\Y)$, $\B(\X)\equiv \B(\X,\X)$, denote the Banach space of bounded linear operator on the Banach space $\X$ to the Banach space $\Y$; ${\|}\cdot {\|}_{\X,\Y}$ denotes the corresponding norm.
\vskip5pt\noindent $\bullet$ $\X\hookrightarrow \Y$ means that $\X\subseteq\Y$ and for any $u\in \X$ there exists $c>0$ such that 
$\|u\|_{\Y}\le c\,\|u\|_{\X}$; we say that $\X$ is continuously embedded into $\Y$. 
\vskip5pt\noindent $\bullet$ $u|\Gamma$ denotes the restriction of the function $u$ to the set $\Gamma$; $L|{\mathsf V}$ denotes the restriction of the linear operator $L$ to the subspace ${\mathsf V}$. 
\vskip5pt\noindent $\bullet$ $H^{s}(\RE^{3})$, $s\in\RE$, denotes the scale of Hilbert space of Sobolev functions on $\RE^{3}$, i.e. $u\in H^{s}(\RE^{3})$ if and only if $k\mapsto (1+\|k\|^{2})^{s/2}\,\widehat u(k)$ is square integrable, $\widehat u$ denoting Fourier transform.
\vskip5pt\noindent $\bullet$ $\Omega\equiv\Omega_{\-}\subset\RE^{3}$ denotes a bounded open set with a Lipschitz boundary $\Gamma$; $\Omega_{\+}:=\RE^{3}\backslash\overline\Omega$.
\vskip5pt\noindent $\bullet$ $\gamma_{0}$ and $\gamma_{1}$ denote the Dirichlet and Neumann traces on the boundary $\Gamma$.
\vskip5pt\noindent $\bullet$ $\Delta^{D}_{\Omega_{\-/\+}}$ denotes the self-adjoint operator in $L^{2}(\Omega_{\-/\+})$ representing the Laplace operator with homogeneous Dirichlet boundary conditions at $\Gamma$.
\vskip5pt\noindent $\bullet$ $\Delta^{N}_{\Omega_{\-/\+}}$ denotes the self-adjoint operator in $L^{2}(\Omega_{\-/\+})$ representing the Laplace operator with homogeneous Neumann boundary conditions at $\Gamma$.
\vskip5pt\noindent $\bullet$ $H^{s}(\Omega_{\-/\+})$, $s\in\RE$, denotes the scale of Hilbert space of Sobolev functions on $\Omega_{\-/\+}$.
\vskip5pt\noindent $\bullet$ $\C^{\kappa}(\Gamma)$ denotes the space of H\"older-continuous functions of order $\kappa$ on $\Gamma$. 
\vskip5pt\noindent $\bullet$ $H^{s}(\Gamma)$, $|s|\le 1$, denotes the Hilbert space of Sobolev functions of order $s$ on $\Gamma$.  
\vskip5pt\noindent $\bullet$ $\M(H^{s}(\Gamma),H^{t}(\Gamma))$, $\M(H^{s}(\Gamma),H^{s}(\Gamma))\equiv \M(H^{s}(\Gamma))$,  denotes the space of Sobolev multipliers from $H^{s}(\Gamma)$ to $H^{t}(\Gamma)$.
\vskip5pt\noindent $\bullet$ $s_{\sharp}$, $\sharp=D,N$, denote the indices 
$\ s_{D}=1/2$, $\ s_{N}=-1/2$.
\vskip5pt\noindent $\bullet$ $\varphi_{n}\rightharpoonup\varphi$ means that the sequence $\{\varphi_{n}\}_{1}^{\infty}$ weakly converges to $\varphi$.
\vskip5pt\noindent $\bullet$  ${\mathsf V}^{\perp}\subseteq\X^{*}$, denotes the annihilator ${\mathsf V}^{\perp}=\{x^{*}\in\X^{*}:\langle x^{*},x\rangle_{\X^{*}\!,\X}=0\ \text{for all $x\in{\mathsf V}$}\}$ of the subspace ${\mathsf V}\subseteq\X$. 
\subsection{Trace maps and layer operators on Lipschitz manifolds.}
Let $\Gamma$ be the compact Lipschitz manifold given by the boundary of $\Omega\subset\RE^{3}$. Let $\gamma_{0}$ be  the map defined by the restriction of $u\in \C^{\infty}_{comp}(\RE^{3})$ along the set  $\Gamma$: $\gamma_{0}u:=u|\Gamma$.  Then, by \cite[Theorem 1, Chapter VII]{JW},  such a map has a bounded and surjective extension to $H^{s+1/2}(\RE^{3})$ for any $s>0$:
\be\label{0trace}
\gamma_{0}:H^{s+1/2}(\RE^{3})\to B_{2,2}^{s}(\Gamma)\,.
\ee
Here the Hilbert space $B^{s}_{2,2}(\Gamma)$ is a Besov-like space (see \cite[Section 2, Chapter V]{JW} for the precise definitions); $B^{s}_{2,2}(\Gamma)$ identifies with $H^{s}(\Gamma)$ whenever 
$0<s<1$ (see \cite[Section 1.1, chap. V]{JW}), where $H^{s}(\Gamma)$ denotes the usual fractional Sobolev space on $\Gamma$ (see e.g. \cite[Chapter 3]{McL}). If $\Gamma$ is a manifold of class $\C^{\kappa,1}$, $\kappa\ge 0$, then $B^{s}_{2,2}(\Gamma)=H^{s}(\Gamma)$ for any $s\le\kappa+1$. We use the following notations for the dual (with respect to the $L^{2}(\Gamma)$-pairing) spaces: $(B^{s}_{2,2}(\Gamma))^{*}\equiv B^{-s}_{2,2}(\Gamma)$.\par
By \cite[Proposition 20.5]{Trib}, the embeddings
$B^{s_{2}}_{2,2}(\Gamma)\hookrightarrow B^{s_{1}}_{2,2}(\Gamma)$, $s_{2}>s_{1}$, and $B^{s}_{2,2}(\Gamma)\hookrightarrow L^{{2}/{(1-s)}}(\Gamma)$,  $0<s<1$, are compact.\par
Let $\Delta:H^{s+2}(\RE^{3})\to H^{s}(\RE^{3})$ be the distributional Laplacian; in the following the resolvent $R^{0}_{z}\equiv(-\Delta+z)^{-1}$, $z\in\CO\backslash(-\infty,0]$, is viewed as a map in $\B(H^{s}(\RE^{n}), H^{s+2}(\RE^{n}))$, $s\in\RE$. Given $s>0$, by the mapping properties  \eqref{0trace} one gets, for the dual of the trace map, 
$$\gamma_{0}^{*}:B_{2,2}^{-s}(\Gamma)\to H^{-s-1/2}(\RE^{3})$$ and so we can define the bounded operator (the single-layer potential)
\be\label{SL}
\SL_{z}:=R^{0}_{z}\gamma_{0}^{*}: B_{2,2}^{-s}(\Gamma)\to H^{3/2-s}(\RE^{3})
\,.
\ee
By resolvent identity one has 
\be\label{Swz}
\SL_{z}-\SL_{w}=(w-z)R^{0}_{z}\SL_{w}\,.
\ee
By \eqref{0trace} and \eqref{SL}, one obtains  the bounded operator
$$
\gamma_{0}\SL_{z}: B_{2,2}^{-s}(\Gamma)\to B_{2,2}^{1-s}(\Gamma)\,.
$$
In the following $\Delta_{\Omega_{\-/\+}}$ denote the distributional Laplacians on $\Omega_{\-/\+}$.
\par 
The one-sided, zero and first order, trace operators $\gamma_{0}^{\-/\+}$ and $\gamma_{1}^{\-/\+}=\nu\cdot\gamma_{0}^{\-/\+}\nabla$ ($\nu $ denoting the outward normal vector at the boundary) defined on
smooth functions in $\mathcal{C}_{comp}^{\infty}(  \overline{\Omega}_{\-/\+
})  $ extend to bounded and surjective linear operators (see e.g. \cite[Theorem
3.38]{McL})
\begin{equation}
\gamma_{0}^{\-/\+}\in{\B}(H^{s+1/2}(  \Omega_{\-/\+})  ,H^{s}(  \Gamma)  )\,,\qquad 0<s< 1\,.
\label{Trace_Gamma_plusmin_est}%
\end{equation}
and
\begin{equation}
\gamma_{1}^{\-/\+}\in{\B}(H^{s+3/2}(  \Omega_{\-/\+})  ,H^{s}(  \Gamma)  )\,,\qquad 0<s<1
\label{Trace_Gamma_plusmin_est_1}%
\end{equation}
(we refer to \cite[Chapter 3]{McL} for the definition of the Sobolev spaces $H^{s}(\Omega_{\-/\+})$ and $H^{s}(\Gamma)$). Using these maps and setting $H^{s}(  \RE^3\backslash\Gamma ):=H^{s}(  \Omega_{\-})  \oplus H^{s}(  \Omega
_{\+})$, the two-sided bounded and surjective trace operators are defined according
to%
\begin{equation}
\gamma_{0}:H^{s+1/2}(  \RE^3\backslash\Gamma )\rightarrow H^{s}(\Gamma)\,,\quad\gamma_{0}%
(u_{\-}\oplus u_{\+}):=\frac{1}{2}(\gamma_{0}^{\-}u_{\-}+\gamma_{0}^{\+}u_{\+})\,,
\label{trace_ext_0}%
\end{equation}%
\begin{equation}
\gamma_{1}:H^{s+3/2}(  \RE^3\backslash\Gamma )\rightarrow H^{s}(\Gamma)\,,\quad\gamma_{1}%
(u_{\-}\oplus u_{\+}):=\frac{1}{2}(\gamma_{1}^{\-}u_{\-}+\gamma_{0}^{\+}u_{\+})\,,
\label{trace_ext_1}%
\end{equation}
while the corresponding jumps are%
\begin{equation}
[\gamma_{0}]:H^{s+1/2}(  \RE^3\backslash\Gamma )  \rightarrow H^{s}(\Gamma)\,,\quad[
\gamma_{0}](u_{\-}\oplus u_{\+}):=\gamma_{0}^{\-}u_{\-}-\gamma_{0}^{\+}u_{\+}\,,
\end{equation}%
\begin{equation}
[\gamma_{1}]:H^{s+3/2}(  \RE^3\backslash\Gamma )\rightarrow H^{s}(\Gamma)\,,\quad[
\gamma_{1}](u_{\-}\oplus u_{\+}):=\gamma_{1}^{\-}u_{\-}-\gamma_{1}^{\+}u_{\+}\,.
\end{equation}
Let us notice that in the case $u=u_{\-}\oplus u_{\+}\in H^{s+1/2}(\RE^{n})$, $0<s< 1$, $\gamma_{0}$ in \eqref{trace_ext_0} coincides with the map defined in \eqref{0trace} and so there is no ambiguity in our notations; this also entails that $\gamma_{0}$ remains  surjective even if restricted to $H^{2}(\RE^{3})$. Similarly the map $\gamma_{1}$ is surjective onto $H^{s}(\Gamma)$ even if restricted to $H^{s+3/2}(\RE^{3})$.\par
By \cite[Lemma 4.3]{McL}, the trace maps $\gamma_{1}^{\-/\+}$ can be extended to the spaces $$H^{1}_{\Delta}(\Omega_{\-/\+}):=\{u_{\-/\+}\in H^{1}(\Omega_{\-/\+}):\Delta_{\Omega_{\-/\+}}u_{\-/\+}\in L^{2}(\Omega_{\-/\+})\}\,:$$ 
$$
\gamma_{1}^{\-/\+}: H^{1}_{\Delta}(\Omega_{\-/\+})\to H^{-1/2}(\Gamma)\,.
$$
This gives the analogous extensions of the maps $\gamma_{1}$ and $[\gamma_{1}]$ defined on $H^{1}_{\Delta}(\RE^{3}\backslash\Gamma):=H^{1}_{\Delta}(\Omega_{\-})\oplus H^{1}_{\Delta}(\Omega_{\+})$ with values in $H^{-1/2}(\Gamma)$. \par
By using a cut-off function $\chi\in \mathcal{C}_{comp}^{\infty}( \RE^{n})$ such that $\chi=1$ in a neighborhood of $\Omega_{\-}$, all the maps defined above can be extended (and we use the same notation) to functions $u$ such that $\chi u$ is in the right function space.\par
The single-layer operator $\SL_{z}$ has been already introduced above; now we recall the definition of double-layer operator $\DL_{z}$, $z\in\CO\backslash(-\infty,0]$: by the dual map
$$
\gamma_{1}^{*}:H^{-s}(\Gamma)\to H^{-s-3/2}(\RE^{3})
$$
and by the resolvent $R^{0}_{z}\in\B(H^{s}(\RE^{3}),H^{s+2}(\RE^{3}))$, one defines the bounded operator
\be\label{DL}
\DL_{z}:H^{-s}(\Gamma)\to H^{-s+1/2}(\RE^{3})\,,\quad \DL_{z}:=R^{0}_{z}(\gamma_{1})^{*}\,,\quad 0<s< 1\,.
\ee
By resolvent identity one has 
\be\label{Dwz}
\DL_{z}-\DL_{w}=(z-w)R^{0}_{z}\DL_{w}\,.
\ee
By the mapping properties of the layer operators, one gets (see \cite[Theorem 6.11]{McL})
\be\label{map}
\chi\SL_{z}\in \B(H^{-1/2}(\Gamma),H^{1}(\RE^{3}))\,,\qquad\chi\DL_{z}\in \B(H^{1/2}(\Gamma),H^{1}(\RE^{3}\backslash\Gamma))\,,
\ee
for any $\chi\in \C^{\infty}_{comp}(\RE^{3})$; by $(-(\Delta_{\Omega_{\-}}\oplus\Delta_{\Omega_{\+}})+z)\SL_{z}\phi=(-(\Delta_{\Omega_{\-}}\oplus\Delta_{\Omega_{\+}})+z)\DL_{z}\varphi=0$, one gets $\chi\SL_{z}\phi\in H^{1}_{\Delta}(\RE^{n}\backslash\Gamma)$, $\phi\in H^{-1/2}(\Gamma)$, and $\chi\DL_{z}\varphi\in H^{1}_{\Delta}(\RE^{n}\backslash\Gamma)$, $\varphi\in H^{1/2}(\Gamma)$. Thus
$$
\gamma_{0}\SL_{z}\in \B(H^{-1/2}(\Gamma),H^{1/2}(\Gamma))\,,\qquad 
\gamma_{1}\DL_{z}\in \B(H^{1/2}(\Gamma),H^{-1/2}(\Gamma))\,.
$$
These mapping properties can be extended to a larger range of Sobolev spaces (see, e.g., \cite[Theorem 6.12 and successive remarks]{McL}): 
$$
\gamma_{0}\SL_{z}\in \B(H^{s-1/2}(\Gamma),H^{s+1/2}(\Gamma))\,,\quad 
\gamma_{1}\DL_{z}\in \B(H^{s+1/2}(\Gamma),H^{s-1/2}(\Gamma))\,,\quad 
-1/2\le s\le 1/2\,.
$$
By the Limiting Absorption Principle for the free Laplacian (see, e.g.,  \cite[Section 18]{KomKop}), duality and interpolation, one has that the limits 
$$
R^{0,\pm}_{\lambda}:=\lim_{\epsilon\downarrow 0}
R^{0}_{\lambda\pm i\epsilon}
$$ 
exist in $\in \B(H^{-s}_{w}(\RE^{3}),H^{-s+2}_{-w}(\RE^{3}))$, $w>1/2$, $0\le s\le 2$ (here $H_{w}^{s}(\RE^{3})$ denotes the weighted Sobolev space of order $s$ with weight $\varphi(x)=(1+\|x\|^{2})^{w/2}$). Thus, since $\Gamma$ is bounded,  
the limits 
\be\label{fLAP}
\SL^{\pm}_{\lambda}:=R^{0,\pm}_{\lambda}\gamma_{0}^{*}=\lim_{\epsilon\downarrow 0}\SL_{\lambda\pm i\epsilon}
\,,\qquad
\DL^{\pm}_{\lambda}:=R^{0,\pm}_{\lambda}\gamma_{1}^{*}=\lim_{\epsilon\downarrow 0}\DL_{\lambda\pm i\epsilon}
\ee
exist in $\B(B_{2,2}^{-s}(\Gamma), H^{3/2-s}_{-w}(\RE^{3}))$, $0<s\le 3/2$, and
$\B(H^{-s}(\Gamma), H^{1/2-s}_{-w}(\RE^{3}))$, $0<s\le 1/2$, respectively. Moreover, by the identities \eqref{Swz},\eqref{Dwz} and by $\SL_{z}\in \B(B_{2,2}^{-3/2}(\Gamma), L^{2}_{w}(\RE^{n}))$, $\DL_{z}\in \B(H^{-1/2}(\Gamma), L^{2}_{w}(\RE^{n}))$ (see \cite[relation (4.10)]{JST}) one has
\be\label{Rlim}
\SL^{\pm}_{\lambda}=\SL_{z}+(z-\lambda)R^{0,\pm}_{\lambda}\SL_{z}\,,\quad 
\DL^{\pm}_{\lambda}=\DL_{z}+(z-\lambda)R^{0,\pm}_{\lambda}\DL_{z}\,.
\ee
\section{Direct Scattering Theory for Singular Perturbations.}
\subsection{Singular Perturbations of the Laplace operator.}
Let $\Delta:H^{2}(\RE^{3})\subseteq L^{2}(\RE^{3})\to L^{2}(\RE^{3})$ be the self-adjoint operator given by the free Laplacian on the whole space. Another  self-adjoint operator $\widetilde\Delta:\dom(\widetilde\Delta)\subseteq L^{2}(\RE^{3})\to L^{2}(\RE^{3})$ is said to
be a singular perturbation of $\Delta$ if the set $${\mathsf D}:=\{u\in H^{2}(\RE^{3})\cap \dom(\widetilde\Delta)\, :\, \Delta u=\widetilde\Delta u\}$$ is dense in
$L^{2}(\RE^{3})$. Our aim is the study of direct and inverse scattering for the couple $(\widetilde\Delta,\Delta)$. Notice that $\widetilde\Delta$ is a self-adjoint extension of the symmetric operator $\Delta^{\!\circ}:=\Delta|{\mathsf D}\equiv \widetilde\Delta|{\mathsf D}$; in typical situations $\widetilde\Delta$ represents the Laplace operator with some kind of boundary condition holding on a null subset. 
\subsection{Wave Operators.}\label{WaveOp} Given the two self-adjoint operators $\Delta$ and $\widetilde\Delta$, let $e^{it\Delta}$ and $e^{it\widetilde\Delta}$ be the corresponding unitary groups of evolution providing solutions of the Cauchy problems for the Schr\"odinger equations 
\be\label{schr}
i\,\frac {du }{dt}=-\Delta u\,,\qquad 
i\,\frac {du}{dt}=-\widetilde\Delta u\,.
\ee
As usual in Quantum Mechanics (see, e.g., \cite{RS}), we define the Wave Operators for the scattering couple $(\widetilde\Delta,\Delta)$ as 
$$
W_{\pm}(\widetilde\Delta,\Delta)u:=\lim_{t\to\mp\infty}e^{-it\widetilde\Delta}e^{it\Delta}\,u\,.
$$
One says that $W_{\pm}(\widetilde\Delta,\Delta)$ exist whenever the limits exist for any vector $u\in L^{2}(\RE^{3})$ and then that are complete whenever
$$
\ran(W_{+}(\widetilde\Delta,\Delta))=:\H_{\rm in}=\H_{\rm out}:=\ran(W_{-}(\widetilde\Delta,\Delta))=L^{2}(\RE^{3})_{\rm ac}\,,
$$ 
where $L^{2}(\RE^{3})_{\rm ac}$  denotes the absolutely continuous subspace of $\widetilde\Delta$. 
It is known that the existence of both the wave operators $W_{\pm}(\widetilde\Delta,\Delta)$ and $W_{\pm}(\Delta,\widetilde\Delta)$  gives completeness. From the point of view of physical interpretation, a more relevant definition is the following:   $W_{\pm}(\widetilde\Delta,\Delta)$ are said to be asymptotically complete whenever they are complete and
$$\H_{\rm in}=\H_{\rm out}=L^{2}(\RE^{3})_{\rm pp}^{\perp}\,,$$ where $L^{2}(\RE^{3})_{\rm pp}$ denotes the pure point subspace of $\widetilde\Delta$; equivalently, whenever 
they are complete and the singular continuous spectrum of $\widetilde\Delta$ is empty: $\sigma_{\rm sc}(\widetilde\Delta)=\emptyset$. In this case $L^{2}(\RE^{3})$ decomposes into the direct sum of scattering states and bound states.  
\subsection{Scattering theory for wave equations.}\label{swe} 
Suppose that  $\widetilde\Delta$ is real (i.e., it maps real-valued functions to real-valued functions), not positive and injective (these hypotheses can be weakened, it suffices to require $\widetilde\Delta$ upper semi-bounded, see \cite[Sections 8 and 9]{Kato2}, \cite[Section 10.3]{BW}). Let $H_{\rm hom}^{1}(\RE^{3})$ be the homogeneous Sobolev space of order one and let $\widetilde H_{\rm hom}^{1}(\RE^{3})$ the completion, with respect to the norm $\|u\|:=\|(-\widetilde\Delta)^{1/2}u\|_{L^{2}(\RE^{3})}$, of $\dom(-\widetilde\Delta)^{1/2})$.
Then the unitary group of evolutions providing the solutions of the Cauchy problems with {\it real} initial conditions 
$$
\begin{cases}
\frac {d\, }{dt}\,u(t)=v(t)\\
\frac {d\, }{dt}\,v(t)=\Delta u(t)\\
u(0)=u_{0}\in H_{\rm hom}^{1}(\RE^{3})\\
 v(0)=v_{0}\in L^{2}(\RE^{3})\,,
\end{cases}\qquad
\begin{cases}
\frac {d\, }{dt}\,\widetilde u(t)=\widetilde v(t)\\
\frac {d\, }{dt}\,\widetilde v(t)=\widetilde\Delta \widetilde u(t)\\
\widetilde u(0)=\widetilde u_{0}\in \widetilde H_{\rm hom}^{1}(\RE^{3})\\ 
\widetilde v(0)=\widetilde v_{0}\in L^{2}(\RE^{3})\,,
\end{cases}
$$ 
are unitary equivalent, by the maps 
$$
u\oplus v\mapsto (-\Delta)^{1/2}u+i\,v\,,\qquad \widetilde u\oplus \widetilde v\mapsto (-\widetilde\Delta)^{1/2}\widetilde u+i\,\widetilde v\,,
$$ 
to the Schr\"odinger unitary groups in the {\it complex} Hilbert space $L^{2}(\RE^{3})$ given by 
$e^{-it(-\Delta)^{1/2}}$ and  $e^{-it(-\widetilde\Delta)^{1/2}}$ respectively. By the Kato-Birman invariance principle (see, e.g., \cite[Section 11.3.3]{BW}), if both the wave operators  $W_{\pm}(\widetilde\Delta,\Delta)$ and 
$W_{\pm}(-(-\widetilde\Delta)^{1/2},-(-\Delta)^{1/2})$ exist, then they are equal (by the Kato-Birman criterion, see \cite[Theorem 4.8, Chapter X]{Kato}, equality holds whenever the difference of some power of the resolvents is trace-class; for the models discussed below this is true under some additional regularity hypotheses on $\Gamma$, see \cite[Theorems 4.11 and 4.12]{JDE}). In this case  
the scattering theory for the couple of Schr\"odinger equations  \eqref{schr} is equivalent to the one for the couple of wave equations 
$$
\frac {d^{2}u }{dt^{2}}=\Delta u\,,\qquad \frac {d^{2}u }{dt^{2}}=\widetilde\Delta u\,.
$$
\subsection{A resolvent formula for singular perturbations.} Given an auxiliary Hilbert space $\K$, we introduce a linear application $\tau:H^{2}(\RE^{3})\to\K$
which plays the role of an abstract trace (evaluation) map. We assume that  \vskip5pt\noindent 
1. $\tau$ is continuous; \par\noindent
2. $\tau$ is surjective (so that $\K$ plays the role of the trace space);\par\noindent 
3. ker$(\tau)$ is dense in $L^{2}(\RE^{3})$. 
\vskip5pt\noindent In the following we do not identify $\K$ with its dual $\K^{*}$; however we use $\K^{**}\equiv\K$. Tipically $\K\hookrightarrow\K_{0}\hookrightarrow\K^{*}$ and the $\K$-$\K^{*}$ duality $\langle\cdot,\cdot\rangle_{\K^{*}\!,\K}$ (conjugate-linear with respect to the first variable) is defined in terms of the scalar product of the Hilbert space $\K_{0}$.  For any $z\in\rho(A_{0})$ we define the bounded operators $$R^{0}_{z}:=(-\Delta+z)^{-1}:L^{2}(\RE^{3})\to H^{2}(\RE^{3})$$  and $$G_{z}:=(\tau R^{0}_{z^{*}})^*:\K^{*}\to L^{2}(\RE^{3})\,.$$
Then, given a reflexive Banach space $\X$ such that $\K\hookrightarrow\X$, we consider, for some not empty set $Z_{\Lambda}\subseteq \CO\backslash(-\infty,0]$ which is symmetric with respect to the real axis (i.e., $z\in Z_{\Lambda}\Rightarrow z^{*}\in Z_{\Lambda})$, a map 
\be\label{mapL}
\Lambda:Z_{\Lambda}\to\B(\X ,\X^*)\,,\qquad z\mapsto\Lambda_{z}\,,
\ee
such that 
\be
\label{Lambda}
\Lambda_{z}^{*}=\Lambda_{z^{*}}\,,
\qquad
\Lambda_{w}-\Lambda_{z}=(z-w)\Lambda_{w}G_{w^{*}}^{*}G_{z}\Lambda_{z}\,.
\ee
\begin{remark}\label{MM} Notice that whenever there exists a family of bijections $M_{z}\in \B(\X^{*},\X)$, $z\in Z_{\Lambda}$, such that  $\Lambda_{z}=M_{z}^{-1}$, then \eqref{Lambda} is equivalent to 
\be\label{M1-M2}
M_{z}^{*}=M_{z^{*}}\,,\qquad M_{z}-M_{w}=(z-w)\,G^{*}_{w^{*}}G_{z}\,.
\ee
\end{remark}
The following result is a useful ingredient in the successive discussion about inverse scattering:
\begin{lemma}\label{Im-not-zero} Let $M_{z}\in \B(\X^{*},\X)$, $z\in Z_{\Lambda}$, satisfy \eqref{M1-M2}. Then
$$
\forall z\in Z_{\Lambda}\cap\CO\backslash\RE\,,\quad\forall \phi\in\X^{*}\backslash\{0\}\,,\qquad \text{\rm Im}\langle\phi,M_{z}\phi\rangle_{\X^{*},\X}\not=0\,.
$$
\end{lemma}
\begin{proof} By \eqref{M1-M2}, one has $\text{Im}\langle\phi,M_{z}\phi\rangle_{\X^{*},\X}
=\text{Im}(z)\,\|G_{z}\phi\|^{2}_{L^{2}(\RE^{3})}$. Since $G_{z}^{*}=\tau R^{0}_{z^{*}}$ is surjective onto $\K$, $G_{z}$ has closed range by the closed range theorem. Hence, see  \cite[Theorem 5.2, page 231]{Kato}, there exists $c>0$ such that $\|G_{z}\phi\|^{2}_{L^{2}(\RE^{3})}\ge c\,\|\phi\|^{2}_{\K^{*}}$. Therefore, whenever Im$(z)\not=0$,
$$
\text{Im}\langle\phi,M_{z}\phi\rangle_{\X^{*},\X}=0\quad\Longrightarrow\quad \|\phi\|_{K^{*}}=0\quad\Longrightarrow\quad\phi=0
$$
and the proof is done. 
\end{proof}
Now we recall the key result about singular perturbations of $\Delta$ (see \cite[Theorem 2.1]{P01}, \cite[Corollary 3.2]{P04}, \cite[Corollary 3.2]{P08}, \cite[Theorem 2.4]{JMPA}):
\begin{theorem}\label{WO} Let $\tau$ and $\Lambda$ be as above. Then the family of bounded linear maps in $L^{2}(\RE^{3})$ 
\be\label{resolvent}
R_{z}^{\Lambda}:=R^0_{z}+G_{z}\Lambda_{z}G^{*}_{z^{*}}\,,\qquad z\in Z_{\Lambda}\,,
\ee
is the resolvent of a self-adjoint operator $\Delta_{\Lambda}$ which is a singular perturbation of $\Delta$. Moreover, $\Delta_{\Lambda}$ is a self-adjoint extension of the closed symmetric operator $\Delta|\ker(\tau)$ and all its self-adjoint extensions (and any singular perturbation of $\Delta$ as well) are of this kind.  
\end{theorem}
\begin{remark} The map $\Lambda: z\mapsto \Lambda_{z}$ introduced in \eqref{mapL} and \eqref{Lambda}  encodes the boundary conditions that the functions belonging to the self-adjointness domain of the corresponding $\Delta_{\Lambda}$ have to satisfy. We refer to the successive Sections \ref{appli} and \ref{appli-screens} below for various explicit examples. Notice that the  properties required in \eqref{Lambda} are necessary for  the operator family $z\mapsto R^{\Lambda}_{z}$ in \eqref{resolvent} to satisfy the first resolvent identity and $(R_{z}^{\Lambda})^{*}=R_{z^{*}}^{\Lambda}$ 
(see \cite[page 113]{P01}).
\end{remark}
Then, building on some results by Schechter conceived for perturbations by a regular potential (see \cite[Section 9.4]{Sch-book}), one gets a completeness criterion for the scattering couple $(\Delta_{\Lambda},\Delta)\,$ (see \cite[Theorem 2.8]{JMPA}):  
\begin{theorem} \label{scattering} 
Suppose that there exists an open subset ${E}\subseteq\RE$ of full measure such that for any open and bounded $I$, $\overline{I}\subset E$, 
\be\label{in1}
\sup_{(\lambda,\epsilon)\in I\times (0,1)}\,\epsilon^{\frac12}\,\|G_{\lambda\pm i\epsilon}\|_{\B(\K^{*},L^{2}(\RE^{3}))}<+\infty\,,
\ee
and
\be\label{in2}
\sup_{(\lambda,\epsilon)\in I\times (0,1)}\,\|\Lambda_{\lambda\pm i\epsilon}\|_{\B(\K,\K^{*})}<+\infty\,.
\ee
Then both the wave operators $W_{\pm}(\Delta_ {\Lambda},\Delta)$ and $W_{\pm}(\Delta,\Delta_{\Lambda})$ exists and are complete.
\end{theorem}
\subsection{The Scattering Matrix.} According to Theorem \ref{WO}, whenever \eqref{in1} and \eqref{in2} hold, the scattering operator 
$$
S_{\Lambda}:=W_{+}(\Delta_{\Lambda},\Delta)^{*}W_{-}(\Delta_{\Lambda},\Delta)\,.
$$
is a well defined unitary map. Given the direct integral representation of $L^{2}(\RE^{3})$ with respect to the spectral measure of  $\Delta$,  i.e. the unitary map (here ${\mathbb S}^{2}$ denotes the $2$-dimensional unitary sphere in $\RE^{3}$)
$$\label{F0}
{\mathcal F}_{0}:L^{2}(\RE^{3})\to L^{2}((-\infty,0);L^{2}({\mathbb S}^{2}))\,,\quad ({\mathcal F}_{0}u)_{\lambda}(\xi)=-\frac{|\lambda|^{1/4}}{2^{1/2}}\,\widehat u(|\lambda|^{1/2}\xi)\,,
$$
which diagonalizes $\Delta$, 
we define the scattering matrix  
$$
S^{\Lambda}_{\lambda}:L^{2}({\mathbb S}^{2})\to L^{2}({\mathbb S}^{2})
$$ 
by the relation 
$$
{\mathcal F}_{0}S_{\Lambda}{\mathcal F}_{0}^{*}u_{\lambda}=S^{\Lambda}_{\lambda}u_{\lambda}
\,.
$$
The scattering matrix is better studied using Limiting Absorption Principle and stationary scattering theory (see, e.g., \cite{Y}). However, for typical scattering couples $(\Delta_{\Lambda},\Delta)$,  the hypotheses required in \cite{Y} are not satisfied. Thus at first  one considers the scattering matrix for the resolvent couple $(R^{\Lambda}_{\mu}, R^{0}_{\mu})$, $\mu\in\rho(\Delta_{\Lambda})\cap (0,+\infty)$, so to exploit the factorized form of the resolvent difference $R^{\Lambda}_{\mu}- R^{0}_{\mu}$ provided by formula \eqref{resolvent}, and then uses the Birman-Kato invariance principle (see \cite[Section 4]{JMPA}). At the end, one obtains the following (see \cite[Theorem 5.1]{JMPA}; notice that in reference \cite{JMPA}, due to a repeated misprint, the $t\to\pm\infty$
limits has to be replaced by the $t\to\mp\infty$ ones)
\begin{theorem}\label{LH} 
Let $\Delta_{\Lambda}$ denote the self-adjoint operator corresponding to $\Lambda=\{\Lambda_{z}\}_{z\in Z_{\Lambda}}$, $\Lambda_{z}\in\B(\X,\X^{*})$, $\K\hookrightarrow\X$. Suppose that: 
\be\label{H1}\text{$\Delta_{\Lambda}$ is bounded from above;}
\ee
\be\label{H2}\text{the embedding $\ran(\Lambda_{\lambda})\hookrightarrow\K^{*}$ is compact for any $\lambda\ge c_{\Lambda}>0$;}
\ee
\be\label{H3}\text{there exists $\chi\in C^{\infty}_{\rm comp}(\RE^{3})$ such that $\tau u=\tau(\chi u)$ for any $u\in H^{2}(\RE^{3})$.}
\ee
Then asymptotic completeness holds for the scattering couple  $(\Delta_{\Lambda},\Delta)$.
Moreover,
$$
\sigma_{\rm ac}(\Delta_{\Lambda})=\sigma_{\rm ess}(\Delta_{\Lambda})=(-\infty,0]\,,\quad 
\sigma_{\rm sc}(\Delta_{\Lambda})=\emptyset \,,
$$
the scattering matrix $S_{\lambda}^{\Lambda}$ is given by 
$$\label{S-tau}
S_{\lambda}^{\Lambda}=\uno-2\pi iL_{\lambda}\Lambda_{\lambda}^{+}L^{*}_{\lambda}\,,\quad \lambda\in E^{-}_{\Lambda}:=(-\infty,0)\backslash\sigma_{\rm p}^{-}(\Delta_{\Lambda})\,,
$$ 
where $\sigma_{\rm p}^{-}(\Delta_{\Lambda}):=(-\infty,0)\cap\sigma_{\rm p}(\Delta_{\Lambda})$ is a (possibly empty) discrete set,  
$$
\Lambda_{\lambda}^{+}:=\lim_{\epsilon\downarrow 0}\Lambda_{\lambda+i\epsilon}\,,\quad\text{the limit existing in $\B(\X,\X^{*})$,}
$$
and
\be\label{Ll}
L_{\lambda}:\X^{*}\to L^{2}({\mathbb S}^{2})\,,\qquad L_{\lambda}\phi(\xi):=\frac1{{2}^{1/2}}\,\frac{|\lambda|^{1/4}}{(2\pi)^{{3}/2}}\langle \tau(\chi u^{\xi}_{\lambda}),\phi\rangle_{\X,\X^{*}}\,,\quad\lambda\in(-\infty,0)\,.
\ee
Here $u^{\xi}_{\lambda}(x)=e^{i\,|\lambda|^{1/2}\xi\cdot x}$ denotes the plane wave with direction $\xi\in{\mathbb S}^{2}$ and wavenumber $|\lambda|^{1/2}$.
\end{theorem}
\begin{remark}\label{lim-inv}
Let $\Lambda_{z}=M_{z}^{-1}$ as in Remark \ref{MM} and suppose that the limit  $M_{\lambda}^{+}:=\lim_{\epsilon\downarrow 0}M_{\lambda+i\epsilon}$ exists  in $\B(\X^{*},\X)$. Then, by Theorem \ref{LH},  
the inverse $(M^{+}_{\lambda})^{-1}$ exists in  $\B(\X,\X^{*})$ and $\Lambda^{+}_{\lambda}=
(M^{+}_{\lambda})^{-1}$.
\end{remark}
\section{Inverse Scattering for the Laplace operator with boundary conditions on Lipschitz surfaces.} 
With reference to Theorem \ref{LH} and given an open, bounded set $\Omega\equiv\Omega_{\-}\subset\RE^{3}$ with a Lipschitz boundary $\Gamma$ and such that  $\Omega_{\+}:=\RE^{3}\backslash\overline\Omega$ is connected, we consider models where the map $\tau:H^{2}(\RE^{3})\to\K$ corresponds to one of the following three different cases:
$$\text{1) \centerline{$\tau= \gamma_{0}\,,\qquad\K=B^{3/2}_{2,2}(\Gamma)\,,\qquad\X=H^{s}(\Gamma)$, $|s|\le 1$;}}
$$
$$\text{2)\centerline{$\tau= \gamma_{1}\,,\qquad \K=H^{1/2}(\Gamma)\,,\qquad\X=H^{s}(\Gamma)$, $-1\le s<1/2$;}}
$$
$$\text{3)\centerline{$\tau=\gamma_{0}\oplus\gamma_{1}\,,\qquad \K=B^{3/2}_{2,2}(\Gamma)\oplus H^{1/2}(\Gamma)\,,\quad\X=H^{s}(\Gamma)\oplus H^{t}(\Gamma)$, $|s|\le 1$, $-1\le t<1/2$.}}
$$
These settings, with suitable choice of the map $\Lambda$, allow to obtain all the self-adjoint extensions of the closed symmetric operator $\Delta|C^{\infty}_{\rm comp}(\RE^{3}\backslash\Gamma)$. In particular, any self-adjoint realization of the Laplace operator with boundary conditions prescribed either on the surface $\Gamma$ or on a relatively open subset $\Sigma\subset\Gamma$ can be defined in one of the above schemes, see \cite[Theorem 4.4]{JDE} for the case of smooth hypersurfaces. In the present framework, Theorem \ref{LH} allows the boundary $\Gamma$ to be Lipschitz; in the applications we give in Sections 4.1 and 5.1 hypothesis \eqref{H3} is always satisfied since $\Omega$ is bounded; hypotheses \eqref{H1} and \eqref{H2} also hold, \eqref{H1} by a direct checking and \eqref{H2} by compact Sobolev embeddings. \par
The results we provide in this section apply to the cases where the boundary conditions are assigned on the whole boundary $\Gamma$. Then $\Delta_{\Lambda}$ can be interpreted  as a model either of an extended obstacle or of a semi-transparent interface supported on $\Gamma$, whose physical properties are encoded by $\Lambda$.\par
Defining the Far Field operator
\be\label{FF}
F^{\Lambda}_{\lambda}:=\frac1{2\pi i}\,(\uno-S_{\lambda}^{\Lambda})\equiv L_{\lambda}\Lambda^{+}_{\lambda}L_{\lambda}^{*}:L^{2}({\mathbb S}^{2})\to L^{2}({\mathbb S}^{2})\,,
\qquad\lambda\in E^{-}_{\Lambda}\,,
\ee
the inverse scattering problem consists in recovering the shape of the obstacle $\Omega$ from the knowledge of $F^{\Lambda}_{\lambda}$, or, equivalently, from knowledge of the scattering matrix $S_{\lambda}^{\Lambda}$.
\par
\begin{notation}\label{not} In the following we refer to the different settings 1) - 3) above by introducing the index $\sharp$, with  $\sharp=D, N, DN$ according to the possible different choices, to label the operators 
$$L^{\sharp}_{\lambda}:{\X^{s}_{\sharp}}^{*}\to L^{2}({\mathbb S}^{2})\,,$$ 
associated to one of the traces $\tau_{D}=\gamma_{0}$, $\tau_{N}=\gamma_{1}$, $\tau_{DN}=\gamma_{0}\oplus\gamma_{1}$, and  the spaces $\X=\X^{s}_{\sharp}$, where
$$
\X_{D}^{s}:=H^{1/2-s}(\Gamma)\,,\quad \X_{N}^{s}:=H^{-1/2-s}(\Gamma) \,,\quad 
\X_{DN}^{s}:=H^{1/2-s}(\Gamma)\oplus H^{-s}(\Gamma)\,,\quad 0\le s\le 1/2\,.
$$
Furthermore the adopt the short-hand notations $s_{\sharp}$, $\sharp=D,N$, to denote the indices 
$\ s_{D}=1/2$, $\ s_{N}=-1/2$. 
\end{notation}
\begin{remark} Since $\X^{0}_{\sharp}\hookrightarrow \X^{s}_{\sharp}$, and hence ${\X^{s}_{\sharp}}^{*}\hookrightarrow {\X^{0}_{\sharp}}^{*}$, we do not put any index $s$ in the notation for $L^{\sharp}_{\lambda}$, since we can always suppose that $L_{\lambda}^{\sharp}$ acts on ${\X^{0}_{\sharp}}^{*}$ and is then restricted to the proper space according to the case.
\end{remark}
\begin{lemma}\label{LL} Let $\lambda\in(-\infty,0)\backslash\sigma_{\rm disc}(\Delta^{\sharp}_{\Omega})$, $\sharp=D,N$, and set
\be\label{phi}
\phi^{x}_{\lambda}:{\mathbb S}^{2}\to\CO\,,\quad \phi^{x}_{\lambda}(\xi):=u^{\xi}_{\lambda}(x)\equiv e^{i\,|\lambda|^{1/2}\xi\cdot x}\,.
\ee
Then 
$$
x\in\Omega\quad\iff\quad \phi^{x}_{\lambda}\in\ran({L^{\sharp}_{\lambda}}|H^{s-s_{\sharp}}(\Gamma))\,,\quad s\in [0,1/2]\,,\quad\ s_{D}=1/2\,,\ s_{N}=-1/2\,.
$$
 \end{lemma}
\begin{proof} Given $\lambda\in(-\infty,0)$, let $u^{\sharp}_{\lambda,\phi}$ be the radiating solution (i.e satisfying the Sommerfeld radiating condition)  in $\Omega_{\+}:=\RE^{3}\backslash\overline\Omega$ of Helmholtz equation $(-\Delta+\lambda)u_{\lambda,\phi}=0$ with either Dirichlet (whenever $\sharp=D$) or Neumann (whenever $\sharp=N$) boundary condition $\phi\in H^{s_{\sharp}}(\Gamma)$. Such a solution is unique in 
\begin{align*}
&H^{1}_{\Delta,loc}(\Omega_{\+}):=\\
&\{u\in{\mathscr D}^{\prime}(\Omega_{\+}): u_{B}\in H^{1}(\Omega_{\+}\cap B),\ \Delta u_{B}\in L^{2}(\Omega_{\+}\cap B)\ \text{for any open ball $B\supset\overline\Omega$}\}\,,
\end{align*} 
where $u_{B}:=u|\Omega_{\+}\cap B$ (see, e.g., \cite[ Theorem 9.11]{McL} for the Dirichlet case and \cite[Exercise 9.5]{McL} for the Neumann case). Then (see, e.g., \cite[Theorem 1.4]{KG}, \cite[Exercise 9.4(iv)]{McL}) there exists a unique $u^{\sharp,\infty}_{\lambda,\phi}\in\C^{\infty}({\mathbb S}^{2})$ such that 
$$ u^{\sharp}_{\lambda,\phi}(x)=\frac{e^{i\,|\lambda|^{1/2}\|x\|}}{4\pi\,\|x\|}\, u^{\sharp,\infty}_{\lambda,\phi}(\hat x)+O(\|x\|^{-2})\quad\text{as $\|x\|\to+\infty$, uniformly in $\hat x:=x/\|x\|$.}
$$
This defines the data-to-pattern operator 
$$K^{\sharp}_{\lambda}:H^{s_{\sharp}}(\Gamma)\to L^{2}({\mathbb S}^{2})\,,\quad K^{\sharp}_{\lambda}\phi:=u^{\sharp,\infty}_{\lambda,\phi}\,.
$$
Introducing the Herglotz operators $
H^{\sharp}_{\lambda}:L^{2}({\mathbb S}^{2})\to H^{s_{\sharp}}(\Gamma)$ defined by 
\be\label{Herg}
H^{D}_{\lambda}:=\gamma_{0}H_{\lambda}\,,\quad H^{N}_{\lambda}:=\gamma_{1}H_{\lambda}\,,\quad 
H_{\lambda}f(x):=\int_{{\mathbb S}^{2}}\phi^{x}_{\lambda}(\xi)f(\xi)\,d\sigma(\xi)\,,\ee
one has 
$$
\langle L_{\lambda}^{\sharp}\phi,f\rangle_{L^{2}({\mathbb S}^{2})}=\frac1{2^{1/2}}\,\frac{|\lambda|^{1/4}}{(2\pi)^{3/2}}\, \langle\phi, H^{\sharp}_{\lambda}f\rangle_{H^{-s_{\sharp}}(\Gamma),H^{s_{\sharp}}(\Gamma)}
=\frac1{2^{1/2}}\,\frac{|\lambda|^{1/4}}{(2\pi)^{3/2}}\, \langle {H^{\sharp}_{\lambda}}^{*}\phi, f\rangle_{L^{2}({\mathbb S}^{2})}\,.
$$
Since, (see \cite[proofs of Theorems 1.15 and 1.26]{KG})
\be\label{Herg1}
(H^{D}_{\lambda})^{*}=K^{D}_{\lambda}\gamma_{0}\SL^{+}_{\lambda}\,,
\qquad (H^{N}_{\lambda})^{*}=K^{N}_{\lambda}\gamma_{1}\DL^{+}_{\lambda}\,,
\ee
one gets
\be\label{LLL}
L^{D}_{\lambda}=\frac1{2^{1/2}}\,\frac{|\lambda|^{1/4}}{(2\pi)^{3/2}}\,K^{D}_{\lambda}\gamma_{0}\SL^{+}_{\lambda}\,,\qquad 
L^{N}_{\lambda}=\frac1{2^{1/2}}\,\frac{|\lambda|^{1/4}}{(2\pi)^{3/2}}\,K^{N}_{\lambda}\gamma_{1}\DL^{+}_{\lambda}\,.
\ee
Since, for any $s\in [0,1/2]$, $$\gamma_{0}\SL^{+}_{\lambda}:H^{s-1/2}(\Gamma)\to H^{s+1/2}(\Gamma)\,,\qquad\lambda\in (-\infty,0)\backslash\sigma_{\rm disc}(\Delta_{\Omega}^{D})\,,
$$  
and $$\gamma_{1}\DL^{+}_{\lambda}:H^{s+1/2}(\Gamma)\to H^{s-1/2}(\Gamma)\,,\qquad\lambda\in (-\infty,0)\backslash\sigma_{\rm disc}(\Delta_{\Omega}^{N})\,,
$$ 
are bijections (by \cite[relations (5.32) and (5.33)]{JMPA} and the regularity results in \cite[Theorem 3]{costa}), one has 
\be\label{RR1}
\ran(L^{\sharp}_{\lambda}|H^{s-s_{\sharp}}(\Gamma))=\ran(K^{\sharp}_{\lambda}|H^{s+s_{\sharp}}(\Gamma))\,,\qquad \lambda\in(-\infty,0)\backslash\sigma_{\rm disc}(\Delta^{\sharp}_{\Omega})
\,.
\ee
Finally, by \cite[Theorems 1.12 and 1.27]{KG} (it is easy to check that the proofs, there given for $s=0$, hold for any $s\in[0,1/2]$), one has 
\be\label{RR2}
x\in\Omega\quad\iff\quad \phi^{x}_{\lambda}\in\ran({K^{\sharp}_{\lambda}}|H^{s+s_{\sharp}}(\Gamma))
\ee
and the thesis is proven.\end{proof}
\begin{corollary}\label{cLL} Let $\lambda\in(-\infty,0)\backslash\left(\sigma_{\rm disc}(\Delta^{D}_{\Omega})\cap \sigma_{\rm disc}(\Delta^{N}_{\Omega})\right)$. Then 
$$
x\in\Omega\quad\iff\quad \phi^{x}_{\lambda}\in\ran({L^{DN}_{\lambda}}|H^{s-1/2}(\Gamma)\oplus H^{t+1/2}(\Gamma))\,,\quad s,t\in[0,1/2]\,.
$$
\end{corollary}
\begin{proof} Let $\lambda\in(-\infty,0)$. Since $(-\Delta+\lambda)\SL^{+}_{\lambda}(x)=(-\Delta+\lambda)\DL^{+}_{\lambda}(x)=0$, $x\in\Omega_{\+}$,  one gets the identities $K^{D}_\lambda\gamma_{0}\SL^{+}_{\lambda}=K^{N}_\lambda\gamma_{1}\SL^{+}_{\lambda}$ and 
$K^{D}_\lambda\gamma_{0}\DL^{+}_{\lambda}=K^{N}_\lambda\gamma_{1}\DL^{+}_{\lambda}$. Thus, given $\phi\oplus\varphi\in H^{s-1/2}(\Gamma)\oplus H^{t+1/2}(\Gamma)$, one has 
\begin{align}\label{L-DN}
L^{DN}_{\lambda}\phi\oplus\varphi=&\frac1{2^{1/2}}\,\frac{|\lambda|^{1/4}}{(2\pi)^{3/2}}\,(K^{D}_{\lambda}\gamma_{0}\SL^{+}_{\lambda}\phi+K^{N}_{\lambda}\gamma_{1}\DL^{+}_{\lambda}\varphi)\nonumber\\
=&\frac1{2^{1/2}}\,\frac{|\lambda|^{1/4}}{(2\pi)^{3/2}}\,
K^{N}_{\lambda}(\gamma_{1}\SL^{+}_{\lambda}\phi+\gamma_{1}\DL^{+}_{\lambda}\varphi)
\\
=&\frac1{2^{1/2}}\,\frac{|\lambda|^{1/4}}{(2\pi)^{3/2}}\,
K^{D}_{\lambda}(\gamma_{0}\SL^{+}_{\lambda}\phi+\gamma_{0}\DL^{+}_{\lambda}\varphi)\nonumber\,.
\end{align}
Therefore the thesis is consequence of \eqref{RR1}, \eqref{RR2} and Lemma \ref{LL}.
\end{proof}
Let us recall the following definitions:
\begin{definition} Let $\Y$ be a reflexive Banach space. $C\in\B( \Y^{*}\!,\Y)$ is said to be:\par\noindent
{\it coercive}, whenever there exists $c>0$ such that 
\be\label{coerc}
\forall\varphi\in\Y^{*}\,,\qquad\big|\langle \varphi,C\varphi\rangle_{ \Y^{*}\!,\Y}\big|\ge c\,\|\varphi\|^{2}_{\Y^{*}} \,;
\ee
{\it positive}, whenever $C=C^{*}$ and there exists $c>0$ such that 
\be\label{pos}
\forall\varphi\in\Y^{*}\,,\qquad\langle \varphi,C\varphi\rangle_{ \Y^{*}\!,\Y}\ge c\,\|\varphi\|^{2}_{\Y^{*}} \,;
\ee
{\it sign-definite}, whenever either $C$ or $-\,C$ is positive. 
\end{definition}
\begin{remark}\label{inv} Let $C\in\B( \Y^{*}\!,\Y)$ be coercive. Then $C^{*}$ is injective and so $\ran(C)$ is dense by $\overline{\ran(C)}=\ker(C^{*})^{\perp}=\Y$.  Since \eqref{coerc} implies $\|C\varphi\|_{\Y}\ge c\,\|\varphi\|_{\Y^{*}}$, $\ran(C)$ is closed by \cite[Theorem 5.2, page 231]{Kato}. Hence $C$ is a continuous bijection and therefore $C^{-1}\in\B(\Y,\Y^{*})$ by the inverse mapping theorem. 
\end{remark}
We also recall the following useful coercivity criterion (see \cite[Lemma 1.17]{KG}; since our statement is slightly different from the original one, for the reader convenience we give a sketch of the proof there provided):
\begin{lemma}\label{crit}
Let $C\in\B( \Y^{*}\!,\Y)$ be such that $\text{\rm Im}\langle\varphi,C\varphi\rangle_{ \Y^{*}\!,\Y}\not=0$ for any $\varphi\in\Y^{*}\backslash\{0\}$. Suppose $C$ has the decomposition $C=C_{\circ}+K$, where $C_{\circ}=C^{*}_{\circ}$ is coercive and $K$ is compact. Then $C$ is coercive. 
\end{lemma} 
\begin{proof}Supposing that $C$ does not satisfy \eqref{coerc}, one gets a sequence $\{\varphi_{n}\}_{1}^{\infty}$, $\|\varphi_{n}\|_{\Y^{*}}=1$, $\varphi_{n}\rightharpoonup\varphi$, such that $\langle\varphi_{n},C\varphi_{n}\rangle_{ \Y^{*}\!,\Y}\to 0$. Since
\begin{align*}
&\langle\varphi_{n}-\varphi,C_{\circ}(\varphi_{n}-\varphi)\rangle_{ \Y^{*}\!,\Y}
=\langle\varphi_{n},(C-K)(\varphi_{n}-\varphi)\rangle_{ \Y^{*}\!,\Y}
-\langle C_{\circ}\varphi,\varphi_{n}-\varphi\rangle_{ \Y\!,\Y^{*}}
\\
=& 
\langle\varphi_{n},C\varphi_{n}\rangle_{ \Y^{*}\!,\Y}
-\langle\varphi_{n},K(\varphi_{n}-\varphi)\rangle_{ \Y^{*}\!,\Y}
-\langle\varphi_{n},C\varphi\rangle_{ \Y^{*}\!,\Y}
-\langle C_{\circ}\varphi,\varphi_{n}-\varphi\rangle_{ \Y\!,\Y^{*}}\end{align*}
and $\|K(\varphi_{n}-\varphi)\|_{\Y}\to 0$, one gets
$$
\RE\ni\lim_{n\to\infty}
\langle(\varphi_{n}-\varphi),C_{\circ}(\varphi_{n}-\varphi)\rangle_{ \Y^{*}\!,\Y}
=-\langle\varphi,C\varphi\rangle_{ \Y^{*}\!,\Y}\,,
$$
i.e., $\text{\rm Im}\langle\varphi,C\varphi\rangle_{ \Y^{*}\!,\Y}=0$, which gives $\varphi=0$.
Thus $\varphi_{n}\rightharpoonup 0$ and the inequality
$$
0<c\le |\langle\varphi_{n},C_{\circ}\varphi_{n}\rangle_{ \Y^{*}\!,\Y}|\le
|\langle\varphi_{n},C\varphi_{n}\rangle_{ \Y^{*}\!,\Y}|+\|K\varphi_{n}\|_{\Y} 
$$
is violated for $n$ sufficiently large. 
\end{proof}
\begin{notation}
$$
E^{-}_{D}:=(-\infty,0)\backslash\sigma_{\rm disc}(\Delta^{D}_{\Omega})\,,\qquad
E^{-}_{N}:=(-\infty,0)\backslash\sigma_{\rm disc}(\Delta^{N}_{\Omega})\,,\qquad 
E^{-}_{DN}:=E^{-}_{D}\cup E^{-}_{N}\,.
$$ 
\end{notation}
The factorized form of the operator $F^{\Lambda}_{\lambda}$, Lemma \ref{LL} and Corollary \ref{cLL} suggest to take into account Kirsch's inf-criterion: 
\begin{theorem}\label{inf} Let $\lambda\in E^{-}_{\sharp}\cap E^{-}_{\Lambda}$, $\sharp=D,N,DN$, and suppose that the Far Field Operator can be factorized as $$F^{\Lambda}_{\lambda} =BCB^{*}\,,
$$ where $C\in\B( \Y^{*}\!,\Y)$, $\Y$ a reflexive Banach space, is coercive and $B\in\B(\Y,L^{2}({\mathbb S}^{2}))$  is such that
\be\label{ran}
\ran(B)=\ran(L^{\sharp}_{\lambda}|{\X^{s}_{\sharp}}^{*})
\ee 
for some $s\in[0,1/2]$. Then
$$
x\in\Omega\iff
\inf_{\substack{\psi\in L^{2}({\mathbb S}^{2})\\ \langle\psi,\phi^{x}_{\lambda}\rangle_{L^{2}({\mathbb S}^{2})}=1}}\left|\langle\psi,F^{\Lambda}_{\lambda}\psi\rangle_{L^{2}({\mathbb S}^{2})}\right|>0\,
$$
where $\phi^{x}_{\lambda}$ is defined in \eqref{phi}.
\end{theorem}
\begin{proof} By \eqref{coerc} and by \cite[Theorem 1.16]{KG}, 
for any $\phi\in L^{2}({\mathbb S}^{2})\backslash\{0\}$, one has
$$
\phi\in\ran(B)\iff
\inf_{\substack{\psi\in L^{2}({\mathbb S}^{2})\\ \langle\psi,\phi\rangle_{L^{2}({\mathbb S}^{2})}=1}}\left|\langle\psi,F^{\Lambda}_{\lambda}\psi\rangle_{L^{2}({\mathbb S}^{2})}\right|>0\,.
$$
The proof is then concluded by \eqref{ran}, Lemma \ref{LL} and Corollary \ref{cLL}. 
\end{proof}
The next results is a key ingredient for obtaining a different identification criterion for the shape of $\Omega$. 
\begin{theorem}\label{comp-norm} Let $\lambda\in E^{-}_{\Lambda}$. Then $F_{\lambda}^{\Lambda}$ is a normal compact operator. 
\end{theorem}
\begin{proof} Since the scattering matrix $S^{\Lambda}_{\lambda}$ is unitary,
$$
4\pi^{2}\left(F^{\Lambda}_{\lambda}(F^{\Lambda}_{\lambda})^{*}-(F^{\Lambda}_{\lambda})^{*}F^{\Lambda}_{\lambda}\right)=(S^{\Lambda}_{\lambda})^{*}S_{\lambda}^{\Lambda}-S_{\lambda}^{\Lambda}(S^{\Lambda}_{\lambda})^{*}=\uno-\uno=0
$$
and so $F^{\Lambda}_{\lambda}$ is normal. By 
$$\nabla u_{\lambda}^{\xi}=i|\lambda|^{1/2}\xi\, u_{\lambda}^{\xi}\,,\qquad\Delta u_{\lambda}^{\xi}=-|\lambda|\, u_{\lambda}^{\xi}\,,
$$ 
and
$$
|u_{\lambda}^{\xi_{1}}(x)-u_{\lambda}^{\xi_{2}}(x)|^{2}=2\left(1-\cos (|\lambda|^{1/2}(\xi_{1}-\xi_{2})\!\cdot\!x)\right)\,,
$$
$$
|\xi_{1}u_{\lambda}^{\xi_{1}}(x)-\xi_{2}u_{\lambda}^{\xi_{2}}(x)|^{2}=2\left(1-\xi_{1}\!\cdot\!\xi_{2}\cos (|\lambda|^{1/2}(\xi_{1}-\xi_{2})\!\cdot\!x)\right)\,,
$$
one gets (here the constant  $c$ changes from line to line) 
\begin{align*}
&|L_{\lambda}\phi(\xi_{1})-L_{\lambda}\phi(\xi_{2})|^{2}\le c\,\|\tau\|^{2}_{\B(H^{2}(\RE^{3}),\X)}\|\chi(u_{\lambda}^{\xi_{1}}-u_{\lambda}^{\xi_{2}})\|^{2}_{H^2(\RE^{3})}\|\phi\|^{2}_{\X^{*}}\\
\le \,&c\left(\|u_{\lambda}^{\xi_{1}}-u_{\lambda}^{\xi_{2}}\|^{2}_{L^{2}(\supp(\chi))}+\|\nabla(u_{\lambda}^{\xi_{1}}-u_{\lambda}^{\xi_{2}})\|^{2}_{L^{2}(\supp(\chi))}+\|\Delta(u_{\lambda}^{\xi_{1}}-u_{\lambda}^{\xi_{2}})\|^{2}_{L^{2}(\supp(\chi))}\right)\|\phi\|^{2}_{\X^{*}}\\
\le \,&c\left(\|u_{\lambda}^{\xi_{1}}-u_{\lambda}^{\xi_{2}}\|^{2}_{L^{2}(\supp(\chi))}+\|\xi_{1}u_{\lambda}^{\xi_{1}}-\xi_{2}u_{\lambda}^{\xi_{2}}\|^{2}_{L^{2}(\supp(\chi))}\right)\|\phi\|^{2}_{\X^{*}}\\
\le\,&c\,|\xi_{1}-\xi_{2}|^{2}\|\phi\|^{2}_{\X^{*}}\\
\le\,& c\,\text{dist}^{2}_{{\mathbb S}^{2}}(\xi_{1},\xi_{2})\,\|\phi\|^{2}_{\X^{*}}\,.
\end{align*}
Therefore  $L_{\lambda}$ is a bounded map with values in the space $\text{Lip}({\mathbb S}^{2})$ of Lipschitz functions and so $L_{\lambda}$ in Theorem \ref{LH} is a compact operator by the compact embedding $\text{Lip}({\mathbb S}^{2})\hookrightarrow L^{2}({\mathbb S}^{2})$. In conclusion, $F^{\Lambda}_{\lambda}=L_{\lambda}\Lambda^{+}_{\lambda}L_{\lambda}^{*}$ is compact since $\Lambda^{+}_{\lambda}$ is bounded.

\end{proof}
\begin{remark}\label{rem-cn} As consequence of Theorem \ref{comp-norm} (and since $\uno -2\pi i\,F^{\Lambda}_{\lambda}$ is unitary), by spectral theory for compact normal operators (see, e.g., \cite[Section 6]{Jorg}), one has 
$$
\sigma
_{\rm disc}(  F_{\lambda}^{\Lambda})=\sigma(  F_{\lambda}^{\Lambda})  \backslash\{  0\}
=\{z^{\Lambda}_{\lambda,k}\}_{1}^{\infty}\subset\left\{z\in\CO\backslash\{0\}:\left|z-\frac{1}{2\pi i}\right|=\frac{1}{2\pi}\right\}\,,\quad \lim_{k\uparrow\infty}z^{\Lambda}_{\lambda,k}=0\,,
$$
and there exists an orthonormal sequence $\{\psi^{\Lambda}_{\lambda,k}\}_{1}^{\infty}\subset L^{2}({\mathbb S}^{2})$ such that for every $\psi\in L^{2}({\mathbb S}^{2})$, 
$$
\psi=\psi_{0}+\sum_{k=1}^{\infty}\langle\psi_{\lambda,k}^{\Lambda},\psi\rangle_{L^{2}({\mathbb S}^{2})}\,\psi_{\lambda,k}^{\Lambda}\,,\quad\text{where $\psi_{0}\in\ker(F_{\lambda}^{\Lambda})$,}
$$
and 
$$F_{\lambda}^{\Lambda}
=\sum_{k=1}^{\infty}z_{\lambda,k}^{\Lambda}\,\psi_{\lambda,k}^{\Lambda}\otimes \psi_{\lambda,k}^{\Lambda}\,.
$$
\end{remark}
\begin{remark}\label{perp} Notice that, by Remark \ref{rem-cn}, $\{\psi^{\Lambda}_{\lambda,k}\}_{1}^{\infty}\subset \ker(F^{\Lambda}_{\lambda})^{\perp}$ and so $\ran(F^{\Lambda}_{\lambda})\subseteq \ker(F^{\Lambda}_{\lambda})^{\perp}$.
\end{remark}
\begin{theorem}\label{sum} Let $F_{\lambda}^{\Lambda}=BCB^{*}$, where $B$ satisfies  \eqref{ran} and $C$, with $\text{\rm Im}\langle\varphi,C\varphi\rangle_{ \Y^{*}\!,\Y}\not=0$ for any $\varphi\in\Y^{*}\backslash\{0\}$, has the decomposition $C=C_{\circ}+K$, where $C_{\circ}$ is sign-definite and $K$ is compact. Then 
$$
x\in\Omega\iff\sum_{k=1}^{\infty}\frac{|\langle\phi^{x}_{\lambda},\psi^{\Lambda}_{\lambda,k}\rangle_{L^{2}({\mathbb S}^{2})}|^{2}}{|z^{\Lambda}_{\lambda,k}|}<+\infty\,
$$
where $\phi^{x}_{\lambda}$ is defined in \eqref{phi}.
\end{theorem}
\begin{proof} Let $P_{0}: L^{2}({\mathbb S}^{2})\to L^{2}({\mathbb S}^{2})$ be the orthogonal projection   
such that $\ran(P_{0})=L_{\perp}^{2}({\mathbb S}^{2}):=\ker(F_{\lambda}^{\Lambda})^{\perp}$.  Then, by Remark \ref{perp}, $F_{\lambda}^{\Lambda}=
P_{0}F_{\lambda}^{\Lambda}P_{0}$; hence $F_{\lambda}^{\Lambda}=P_{0}BC B^{*}P_{0}=(P_{0}B)C(P_{0}B)^{*} $, and so, by \cite[Theorem 1.16]{KG}, $\ran(B)=\ran(P_{0}B)$. Let $\widetilde F_{\lambda}^{\Lambda}:L_{\perp}^{2}({\mathbb S}^{2})\to L_{\perp}^{2}({\mathbb S}^{2})$ be the injective normal compact operator given by the compression of $F^{\Lambda}_{\lambda}$ to $L_{\perp}^{2}({\mathbb S}^{2})$. By Remark \ref{rem-cn}, $\{\psi^{\Lambda}_{\lambda,k}\}_{1}^{\infty}$ is an orthonormal basis in $L_{\perp}^{2}({\mathbb S}^{2})$ and $\widetilde F_{\lambda}^{\Lambda}
=\sum_{k=1}^{\infty}z_{\lambda,k}^{\Lambda}\,\psi_{\lambda,k}^{\Lambda}\otimes \psi_{\lambda,k}^{\Lambda}$. By functional calculus for normal operators, using the factorization of  $z\in\CO\backslash\{0\}$ given by $z={|z|^{1/2}}\,\sgn(z){|z|^{1/2}}$, $\sgn(z):=|z|^{-1}z$, one gets 
$$
\widetilde F_{\lambda}^{\Lambda}=|\widetilde F_{\lambda}^{\Lambda}|^{1/2}\,
\sgn(\widetilde F_{\lambda}^{\Lambda})\,| \widetilde F_{\lambda}^{\Lambda}|^{1/2}\,.
$$
Since $\widetilde F_{\lambda}^{\Lambda}=\widetilde BC{\widetilde B}^{*}$, where $\widetilde B:=P_{0}B$ (here $P_{0}$ means the surjection $P_{0}:L^{2}({\mathbb S}^{2})\to L_{\perp}^{2}({\mathbb S}^{2})$), by \cite[Theorem 1.23]{KG}, $\ran(|\widetilde F_{\lambda}^{\Lambda}|^{1/2})=\ran(\widetilde B)=\ran(P_{0}B)=\ran(B)$.
Hence $\ran(|\widetilde F^{\Lambda}_{\lambda}|^{1/2})=\ran(L_{\lambda}^{\sharp}|{\X^{s}_{\sharp}}^{*})$ and so, by Lemma \ref{LL} and Corollary \ref{cLL}, $x\in\Omega$ if and only if $\phi^{x}_{\lambda}\in \ran(|\widetilde F^{\Lambda}_{\lambda}|^{1/2})$, equivalently if and only if $\phi^{x}_{\lambda}\in \dom(|\widetilde F^{\Lambda}_{\lambda}|^{-1/2})$.\par  Since $|\widetilde F^{\Lambda}_{\lambda}|^{-1/2}=\sum_{k=1}^{\infty}|z_{\lambda,k}^{\Lambda}|^{-1/2}\,\psi_{\lambda,k}^{\Lambda}\otimes \psi_{\lambda,k}^{\Lambda}$,    
$\phi^{x}_{\lambda}\in \dom(|\widetilde F^{\Lambda}_{\lambda}|^{-1/2})$ if and only if the series $\sum_{k=1}^{\infty}|z_{\lambda,k}^{\Lambda}|^{-1}|\langle\phi^{x}_{\lambda},\psi^{\Lambda}_{\lambda,k}\rangle_{L^{2}({\mathbb S}^{2})}|^{2}$ converges.
\end{proof}
In applications to concrete models, the following consequence of Theorems \ref{inf} and \ref{sum} turns out to be useful:
\begin{theorem}\label{cor} Let $$F^{\Lambda}_{\lambda}=L^{\sharp}_{\lambda}\Lambda^{+}_{\lambda}{L_{\lambda}^{\sharp}}^{*}\,,\quad \lambda\in E^{-}_{\sharp}\cap E^{-}_{\Lambda}\,,\quad\sharp=D,N,DN\,,
$$ 
and suppose that  $\Lambda_{\lambda}^{+}=(M^{+}_\lambda)^{-1}$, where the bijection $M^{+}_\lambda\in\B({\X^{s}_{\sharp}}^{*},{\X^{s}_{\sharp}})$, $s\in[0,1/2]$, has the decomposition $M^{+}_{\lambda}=M^{+}_{\circ}+K^{+}_\lambda$, with $M^{+}_{\circ}$ sign-definite and $K^{+}_{\lambda}$ compact. Then  
$$
x\in\Omega\quad\iff\inf_{\substack{\psi\in L^{2}({\mathbb S}^{2})\\ \langle\psi,\phi^{x}_{\lambda}\rangle_{L^{2}({\mathbb S}^{2})}=1}}\left|\langle\psi,F^{\Lambda}_{\lambda}\psi\rangle_{L^{2}({\mathbb S}^{2})}\right|>0\iff \sum_{k=1}^{\infty}\frac{|\langle\phi^{x}_{\lambda},\psi^{\Lambda}_{\lambda,k}\rangle_{L^{2}({\mathbb S}^{2})}|^{2}}{|z^{\Lambda}_{\lambda,k}|}<+\infty\,,
$$
where the sequences $\{z^{\Lambda}_{\lambda,k}\}_{1}^{\infty}\subset\CO\backslash\{0\}$ and $\{\psi^{\Lambda}_{\lambda,k}\}_{1}^{\infty}\subset L^{2}({\mathbb S}^{2})$ provide the spectral resolution of $F^{\Lambda}_{\lambda}$ as in Remark \ref{rem-cn} and $\phi^{x}_{\lambda}$ is defined in \eqref{phi}.
\end{theorem}
\begin{proof} Let us consider the factorization $F^{\Lambda}_{\lambda}=\big(L^{\sharp}_{\lambda}(M^{+}_\lambda)^{-1}\big)(M^{+}_{\lambda})^{*}\big(L^{\sharp}_{\lambda}(M^{+}_\lambda)^{-1}\big)^{*}$. Then the thesis is consequence of Lemma \ref{crit}, Theorems \ref{inf} and \ref{sum} once one shows that 
$$\text{ $\text{\rm Im}\langle\phi,(M^{+}_{\lambda})^{*}\phi\rangle_{{\X^{s}_{\sharp}}^{*}\!,{\X^{s}_{\sharp}}}
\not=0$ for any $\phi\in{\X^{s}_{\sharp}}^{*}\backslash\{0\}$.}
$$ Equivalently, let us prove that  $\text{\rm Im}\langle\phi,(M^{+}_{\lambda})^{*}\phi\rangle_{{\X^{s}_{\sharp}}^{*}\!,{\X^{s}_{\sharp}}}
=0$ implies $\phi=0$ (our reasonings below are inspired by the ones given in \cite[page 51]{KG}). 
By the definition of $F^{\Lambda}_{\lambda}$ and since $S^{\Lambda}_{\lambda}$ is unitary, one gets
$$
F^{\Lambda}_{\lambda}-(F^{\Lambda}_{\lambda})^{*}=-2\pi i\,(F^{\Lambda}_{\lambda})^{*}F^{\Lambda}_{\lambda}\,.
$$
Setting $B_{\lambda}:=L^{\sharp}_{\lambda}(M^{+}_\lambda)^{-1}$, this gives the identity
\be\label{identity}
\begin{split}
&\text{\rm Im}\langle B_{\lambda}^{*}\psi,(M^{+}_{\lambda})^{*}B^{*}_{\lambda}\psi\rangle
_{{\X^{s}_{\sharp}}^{*}\!,{\X^{s}_{\sharp}}}=\text{\rm Im}\langle \psi,B_{\lambda} (M^{+}_{\lambda})^{*}B^{*}_{\lambda}\psi\rangle_{ L^{2}({\mathbb S}^{2})}\\=&\frac1{2i}\,\langle \psi,({F^{\Lambda}_{\lambda}}-(F^{\Lambda}_{\lambda})^{*})\psi\rangle_{ L^{2}({\mathbb S}^{2})} 
=-\pi\,\|F^{\Lambda}_{\lambda}\psi\|^{2}_{L^{2}({\mathbb S}^{2})}\,.
\end{split}
\ee
Let $\sharp=N,D$; then by \eqref{LLL}, $\ker(B_{\lambda})=\ker(K^{\sharp}_{\lambda})$; hence, by \cite[Lemma 1.13 and Theorem 1.26(b)]{KG}, one has $\ker(B_{\lambda})=\{0\}$ and so $\ran(B^{*}_{\lambda})$ is dense. Let $\sharp=DN$; then, by \eqref{L-DN}, $\ker(B_{\lambda})=\ker(L^{ND})=\ker(K^{D}_{\lambda}\gamma_{0}G^{+}_{\lambda})$, where $G^{+}_{\lambda}(\phi\oplus\varphi):=(\SL^{+}_{\lambda}\phi+\DL^{+}_{\lambda}\varphi)$. Since $G^{+}_{\lambda}(\phi\oplus\varphi)$ is a radiating solution of the Helmholtz equation in $\Omega_{\+}$, $\gamma_{0}G^{+}_{\lambda}(\phi\oplus\varphi)=0$ implies $G^{+}_{\lambda}(\phi\oplus\varphi)=0$. Hence $\ker(B_{\lambda})=\ker(G^{+}_{\lambda})$. Since $G_{\lambda+ i\epsilon}$ converges to $G_{\lambda}^{+}$ in $\B(B^{-3/2}_{2,2}(\Gamma)\oplus H^{1/2}(\Gamma),L^{2}_{w}(\RE^{3}))$ (see \eqref{fLAP}) and 
there exists $c>0$ such that, for any $\epsilon>0$, $$\|G_{\lambda+i\epsilon}(\phi\oplus\varphi)\|_{L^{2}_{w}(\RE^{3})}
\ge c\,\|\phi\oplus\varphi\|_{B^{-3/2}_{2,2}(\Gamma)\oplus H^{1/2}(\Gamma)}
$$
(see \cite[proof of Lemma 3.6]{JMPA}), $G^{+}_{\lambda}$ is injective and so $\ran(B_{\lambda})$ is dense whenever $\sharp=DN$ as well.\par
Let $\phi\in {\X^{s}_{\sharp}}^{*}$ be such that $\text{\rm Im}\langle\phi,(M^{+}_{\lambda})^{*}\phi\rangle
_{{\X^{s}_{\sharp}}^{*}\!,{\X^{s}_{\sharp}}}=0$; let $\{\psi_{n}\}_{1}^{\infty}\subset L^{2}({\mathbb S}^{2})$  be a sequence such that $B^{*}_{\lambda}\psi_{n}\to \phi$. Then, by \eqref{identity}, $F^{\Lambda}_{\lambda}\psi_{n}\to 0$ and so, for any $\psi\in L^{2}({\mathbb S}^{2})$, 
$$
\langle B_{\lambda}^{*}\psi,(M^{+}_{\lambda})^{*}B^{*}_{\lambda}\psi_{n}\rangle
_{{\X^{s}_{\sharp}}^{*}\!,{\X^{s}_{\sharp}}}=\langle\psi, F^{\Lambda}_{\lambda}\psi_{n}\rangle_{L^{2}({\mathbb S}^{2})}
\to \langle B_{\lambda}^{*}\psi,(M^{+}_{\lambda})^{*}\phi\rangle
_{{\X^{s}_{\sharp}}^{*}\!,{\X^{s}_{\sharp}}}=0\,.$$
Therefore $(M^{+}_{\lambda})^{*}\phi\in \ran(B_{\lambda}^{*})^{\perp}=\{0\}$. Since $M^{+}_{\lambda}$ is a bijection, $(M^{+}_{\lambda})^{*}$ is injective and so $\phi=0$.\end{proof}
\begin{remark} If $M^{+}_{\lambda}$ in Theorem \ref{cor}  is merely coercive, then the ``inf'' criterion still holds.
\end{remark}
\subsection{Applications}\label{appli}
\subsubsection{Dirichlet obstacles.}\label{Dir-ob} Let $\Delta^{D}_{\Omega_{\-/\+}}$ denote the self-adjoint operators in $L^{2}(\Omega_{\-/\+})$ corresponding to the Laplace operator with Dirichlet boundary conditions. One has $\Delta^{D}_{\Omega_{\-}}\oplus \Delta^{D}_{\Omega_{\+}}= \Delta_{\Lambda^{D}}$, where $\Lambda^{D}_{z}=-(\gamma_{0}\SL_{z})^{-1}\in\B(H^{1/2}(\Gamma),H^{-1/2}(\Gamma))$, $z\in\CO\backslash(-\infty,0]$, and Theorem \ref{LH} holds in this case (see \cite[Section 5.2]{JMPA}). 
By first resolvent identity,  $ \gamma_{0}\SL_{\lambda}^{+}:H^{-1/2}(\Gamma)\to 
H^{1/2}(\Gamma)$ can be additively decomposed as $\gamma_{0}\SL_{\lambda}^{+}=
\gamma_{0}\SL_{\mu}+(\lambda-\mu)\gamma_{0}R^{0,+}_{\lambda}\SL_{\mu}$, $\mu>0$, 
with $\gamma_{0}\SL_{\mu}$ positive (see \cite[Lemma 3.2]{JDE}) and $\gamma_{0}R^{0,+}_{\lambda}\SL_{\mu}$ compact (see \cite[Section 5.1.3]{JMPA}).
Thus Theorem \ref{cor}  applies to $F^{\Lambda^{D}}_{\lambda}$, $\lambda\in E^{-}_{D}$. 
\subsubsection{Neumann obstacles.}\label{Neu-ob} Let $\Delta^{N}_{\Omega_{\-/\+}}$ denote the self-adjoint operators in $L^{2}(\Omega_{\-/\+})$ corresponding to the Laplace operator with Neumann boundary conditions. One has $\Delta^{N}_{\Omega_{\-}}\oplus \Delta^{N}_{\Omega_{\+}}= \Delta_{\Lambda^{N}}$, where $\Lambda^{N}_{z}=-(\gamma_{1}\DL_{z})^{-1}\in\B(H^{-1/2}(\Gamma),H^{1/2}(\Gamma))$, $z\in\CO\backslash(-\infty,0]$, and Theorem \ref{LH} holds in this case (see \cite[Section 5.3]{JMPA}). 
By first resolvent identity,  $ \gamma_{1}\DL_{\lambda}^{+}:H^{1/2}(\Gamma)\to 
H^{-1/2}(\Gamma)$ can be additively decomposed as $\gamma_{1}\DL_{\lambda}^{+}=
\gamma_{1}\DL_{\mu}+(\lambda-\mu)\gamma_{1}R^{0,+}_{\lambda}\DL_{\mu}$, $\mu>0$, 
with $-\gamma_{1}\DL_{\mu}$ positive (see \cite[Lemma 3.2]{JDE}) and $\gamma_{1}R^{0,+}_{\lambda}\DL_{\mu}$ compact (see \cite[Section 5.1.3]{JMPA}).
Thus Theorem \ref{cor}  applies to $F^{\Lambda^{N}}_{\lambda}$, $\lambda\in E^{-}_{N}$.
\subsubsection{Obstacles with semitransparent boundary conditions $\alpha\gamma_{0}u=[\gamma_{1}]u$.}\label{delta-ob} Here $\alpha$ is a real-valued function and we use the same symbol to denote the corresponding multiplication operator.
\begin{lemma}\label{delta} 1) If $\alpha\in L^{6}(\Gamma)$ and $\frac1\alpha\in L^{\infty}(\Gamma)$ then $(\frac{\uno}\alpha+ \gamma_{0}\SL_{z})^{-1}\in\B(L^2(\Gamma))$, $z\in\CO\backslash\RE$. 2) If both $\alpha$ and $\frac1{\alpha}$ belong to 
$L^{\infty}(\Gamma)$ and $\sgn(\alpha)$ is constant, then $\frac{\uno}\alpha+ \gamma_{0}\SL_{\lambda}^{+}$, $\lambda\in(-\infty,0)$, is the sum of a sign-definite operator plus a compact one. 
\end{lemma}
\begin{proof} Since $L^{6}(\Gamma)\subseteq\M(H^{2/3}(\Gamma), L^{2}(\Gamma))$  and $ \gamma_{0}\SL_{z}\in\B(H^{t-1/2}(\Gamma),H^{t+1/2}(\Gamma))$, $0< t\le 1/2$ (see \cite[equation (5.27)]{JMPA}, one has that   $\uno+\alpha \gamma_{0}\SL_{z}\in\B(L^2(\Gamma))$ and it is injective since  it is invertible (and hence injective) as a map in $H^{-1/3}(\Gamma)$ (use \cite[Lemma 5.8]{JMPA}). Let us now suppose that it is not surjective from $L^{2}(\Gamma)$ onto itself, i.e. we suppose that there exists $\psi\in L^{2}(\Gamma)$ such that $\psi=\phi+\alpha \gamma_{0}\SL_{z}\phi$ with $\phi\in H^{-1/3}(\Gamma)$, and $\phi\notin L^{2}(\Gamma)$. 
Hence $\alpha \gamma_{0}\SL_{z}\phi\notin L^{2}(\Gamma)$, which is not possible since $\SL_{z}\phi\in H^{2/3}(\Gamma)$ and $\alpha\in\M(H^{2/3}(\Gamma), L^{2}(\Gamma))$. In conclusion  $\uno+\alpha \gamma_{0}\SL_{z}\in\B(L^2(\Gamma))$ is a bounded bijection in $L^{2}(\Gamma)$ and so  $(\uno+\alpha \gamma_{0}\SL_{z})^{-1}\in\B(L^2(\Gamma))$ by the inverse mapping theorem. Since $\alpha$ is a.e. finite, $\frac\uno\alpha:L^{2}(\Gamma)\to L^{2}(\Gamma)$ is a continuous bijection. Hence  $\frac{\uno}\alpha+ \gamma_{0}\SL_{z}=\frac1\alpha(\uno+\alpha \gamma_{0}\SL_{z})$ is a continuous bijection and so  $(\frac\uno\alpha+ \gamma_{0}\SL_{z})^{-1}\in\B(L^2(\Gamma))$ by the inverse mapping theorem.\par
Since $ \gamma_{0}\SL_{\lambda}^{+}$ maps $L^{2}(\Gamma)$ onto $H^{1}(\Gamma)$, by the compact embedding $H^{1}(\Gamma)\hookrightarrow L^{2}(\Gamma)$, it is compact. Since 
$\langle\varphi,\frac1{|\alpha|}\,\varphi\rangle_{L^{2}(\Gamma)}\ge \|\alpha\|_{L^{\infty}(\Gamma)}^{-1}{\|\varphi\|^{2}_{L^{2}(\Gamma)}}$ and $\sgn(\alpha)$ is constant, $\frac1\alpha$ is sign-definite. 
\end{proof}
We consider the self-adjoint operator $\Delta_{\Lambda^{\alpha}}$, where 
\be\label{Malpha}
\Lambda_{z}^{\alpha}=(M^{\alpha}_{z})^{-1}\,, \quad z\in\CO\backslash\RE\,,\qquad
M^{\alpha}_{z}:=-\left(\frac{\uno}\alpha+ \gamma_{0}\SL_{z}\right)\in\B(L^2(\Gamma))\,.
\ee
$\Lambda_{z}^{\alpha}$ is well-defined, i.e., $M^{\alpha}_{z}$ has a bounded inverse, by Lemma \ref{delta}. By \cite[Theorem 2.19]{CFP}, the map $z\mapsto \Lambda^{\alpha}_{z}$ and the resolvent formula 
\eqref{resolvent} extend to  $Z_{\Lambda^{\alpha}}:=\rho(\Delta_{\Lambda^{\alpha}})\cap \CO\backslash(-\infty,0]$. 
$\Delta_{\Lambda^{\alpha}}$ provides a self-adjoint realization of the (bounded form above) Laplacian on $\RE^{3}\backslash\Gamma$ with the  semi-transparent boundary conditions at $\Gamma$ given by  $\alpha\gamma_{0}u=[\gamma_{1}]u$, $[\gamma_{0}]u=0$; moreover Theorem \ref{LH} holds in this case (see \cite[Corollary 5.12]{JMPA}). By point 2 in Lemma \ref{delta}, Theorem \ref{cor}  applies to $F^{\Lambda^{\alpha}}_{\lambda}$, $\lambda\in E^{-}_{D}$ (here $E^{-}_{\Lambda^{\alpha}}=(-\infty,0)$ by \cite[Remark 3.8]{JST}).

\subsubsection{Obstacles with semitransparent boundary conditions $\gamma_{1}u=\theta[\gamma_{0}]u$.}\label{delta'-ob} Here $\theta$ is a real-valued function and we use the same symbol to denote the corresponding multiplication operator. 
\begin{lemma}\label{delta'} Let $\theta\in L^{p}(\Gamma)$, $p>2$. Then 1) $(\theta- \gamma_{1}\DL_{z})^{-1}\in\B(H^{-1/2}(\Gamma),H^{1/2}(\Gamma))$, $z\in\CO\backslash\RE$ ; 2) $\theta- \gamma_{1}\DL^{+}_{\lambda}$, $\lambda\in(-\infty,0)$, can be decomposed as the sum of a compact operator plus a  sign-definite one.
\end{lemma}
\begin{proof} Point 1 is consequence of \cite[Lemma 5.14]{JMPA}. Since  $L^{1/s}(\Gamma)\subseteq\M(H^{s}(\Gamma),H^{-s}(\Gamma))$, $s\in[0,1]$, the map $\theta:H^{1/2}(\Gamma)\to H^{-1/2}(\Gamma)$ is compact by the compact embedding $H^{-1/p}(\Gamma)\hookrightarrow H^{-1/2}(\Gamma)$. The difference $ \gamma_{1}\DL_{\lambda}^{+}- \gamma_{1}\DL_{\mu}$ is compact for any $\mu>0$ (see \cite[Section 5.1.3]{JMPA}) and $ -\gamma_{1}\DL_{\mu}$ is  positive (see \cite[Lemma 3.2]{JDE}). 
\end{proof}
We consider the self-adjoint operator 
$\Delta_{\Lambda^{{\theta}}}$, 
\be\label{Mtheta}
\Lambda_{z}^{{\theta}}=(M^{\theta}_{z})^{-1}\,,\quad z\in\CO\backslash\RE\,,\qquad M^{\theta}_{z}:={\theta} -\gamma_{1}\DL_{z}\in\B(H^{1/2}(\Gamma),H^{-1/2}(\Gamma))\,.
\ee
$\Lambda_{z}^{\theta}$ is well-defined, i.e., $M^{\theta}_{z}$ has a bounded inverse, by Lemma \ref{delta'}. By \cite[Theorem 2.19]{CFP}, the map $z\mapsto \Lambda^{\theta}_{z}$ and the resolvent formula 
\eqref{resolvent} extend to $Z_{\Lambda^{\theta}}:=\rho(\Delta_{\Lambda^{\theta}})\cap \CO\backslash(-\infty,0]$. $\Delta_{\Lambda^{{\theta}}}$ provides a self-adjoint realization of the (bounded form above) Laplacian on $\RE^{3}\backslash\Gamma$ with the semi-transparent boundary conditions at $\Gamma$ given by $\gamma_{1}u=\theta[\gamma_{0}]u$, $[\gamma_{1}]u=0$; moreover Theorem \ref{LH} holds in this case (see \cite[Section 5.5]{JMPA}).  By point 2 in Lemma \ref{delta'}, Corollary \ref{crit} applies to $
F^{\Lambda^{{\theta}}}_{\lambda}$, $\lambda\in E_{N}^{-}$ (here $E^{-}_{\Lambda^{\theta}}=(-\infty,0)$ by \cite[Remark 3.8]{JST}).

\subsubsection{Obstacles with local boundary  conditions.}\label{rob-ob}
\begin{lemma}\label{rob}
Let $b_{11}$ and $b_{22}$ real-valued, $b_{11}<0$, $b_{11}\in L^{\infty}(\Gamma)$, $b^{-1}_{11}\in L^{\infty}(\Gamma)$, $b_{22}\in L^{p}(\Gamma)$, $p>2$, $b_{12}\in\C^{\kappa}(\Gamma)$ for some $\kappa\in(0,1)$. Then
$$
M^{b}_{z}:L^{2}(\Gamma)\oplus H^{1/2}(\Gamma)\to L^{2}(\Gamma)\oplus H^{-1/2}(\Gamma)\,,\qquad z\in\CO\backslash(-\infty,0]\,,
$$
$$
M^{b}_{z}:=\left[\begin{matrix} b_{11}+ \gamma_{0}\SL_{z}&b_{12}+ \gamma_{0}\DL_{z}\\
b_{12}^{*}+ \gamma_{1}\SL_{z}&b_{22}+ \gamma_{1}\DL_{z}\end{matrix}\right]
$$
is coercive for any $z\in\CO\backslash\RE$.
\end{lemma} 
\begin{proof} Given $\mu>0$, let us consider the decomposition $M^{b}_{z}=M_{(1)}+M_{(2)}+M_{(3)}$, where
$$
M_{(1)}=\left[\begin{matrix} b_{11}&0\\
0& \gamma_{1}\DL_{\mu}\end{matrix}\right]\,,
$$
$$
M_{(2)}=
\left[\begin{matrix}  \gamma_{0}\SL_{z}&0\\
0&b_{22}+ \gamma_{1}\DL_{z}- \gamma_{1}\DL_{\mu}\end{matrix}\right]\,,
$$
$$
M_{(3)}=
\left[\begin{matrix}0 &b_{12}+ \gamma_{0}\DL_{z}\\
b_{12}^{*}+ \gamma_{1}\SL_{z}&0\end{matrix}\right]\,.
$$
By \cite[Lemma 3.2]{JDE}, $$-\langle\varphi,\gamma_{1}\DL_{\mu}\varphi\rangle_{H^{1/2}(\Gamma),H^{-1/2}(\Gamma)}\ge c_{\mu}\,\|\varphi\|^{2}_{H^{-1/2}(\Gamma)}\,,\quad c_{\mu}>0\,.$$ 
Hence, since $b_{11}<0$,
\begin{align*}
-\left(\langle\phi,b_{11}\phi\rangle_{L^2(\Gamma)}+\langle\varphi,\gamma_{1}\DL_{\mu}\varphi\rangle_{H^{1/2}(\Gamma),H^{-1/2}(\Gamma)}\right)
\ge \|b_{11}^{-1}\|^{-1}_{L^{\infty}(\Gamma)}\|\phi\|^{2}_{L^{2}(\Gamma)}+c_{\mu}\,\|\varphi\|^{2}_{H^{-1/2}(\Gamma)}
\end{align*}
and  $M_{(1)}$ is sign-definite. $M_{(2)}$ is compact
since both its diagonal elements are compact (here one argues as in the proofs of Lemmata \ref{delta} and \ref{delta'}).  Since $\C^{\kappa}(\Gamma)\subseteq\M(H^{s}(\Gamma))\subseteq L^{\infty}(\Gamma)$, $0<s<\kappa$, and $\gamma_{0}\DL_{z}\in\B(H^{1/2}(\Gamma))$, $\gamma_{1}\SL_{z}\in\B(L^{2}(\Gamma))$ (see, e.g., \cite[Theorem 6.12 and successive remarks]{McL}), one has that $M_{(3)}$ maps $L^{2}(\Gamma)\oplus H^{1/2}(\Gamma)$ into 
$H^{s}(\Gamma)\oplus L^{2}(\Gamma)$ for any $s\in[0,1/2]\cap[0,\kappa]$; hence it is compact by the compact embeddings $H^{s}(\Gamma)\hookrightarrow L^{2}(\Gamma)$, $s>0$, and $L^{2}(\Gamma)\hookrightarrow H^{-1/2}(\Gamma)$. Therefore $M^{b}_{z}$ decomposes as the sum of a sign-definite operator plus a compact one. Since, by resolvent identity, $M^{b}_{z}$ satisfies \eqref{M1-M2}, the proof is then concluded by Lemmata \ref{Im-not-zero}  and \ref{crit}.\end{proof}
By Lemma \ref{rob} and Remark \ref{inv}, the operator-valued map $z\mapsto\Lambda^{b}_{z}$, 
$\Lambda^{b}_{z}:=(M_{z}^{b})^{-1}$, $z\in \CO\backslash\RE$, is well defined and, by \eqref{M1-M2}, satisfies \eqref{Lambda}. Therefore, by Theorem \ref{WO}, we can define the self-adjoint operator $\Delta_{\Lambda^{b}}$; it provides a self-adjoint realization of the Laplacian on $\RE^{3}\backslash\Gamma$ with boundary conditions 
$$\begin{cases}
\gamma_{0}u=b_{11}[\gamma_{0}]u+b_{12}[\gamma_{1}]u\,,\quad\\
\gamma_{1}u=b^{*}_{12}[\gamma_{0}]u+b_{22}[\gamma_{1}]u
\end{cases}$$ 
(see \cite[Corollary 4.9]{JDE}). By \cite[Theorem 2.19]{CFP}, the map $z\mapsto \Lambda^{b}_{z}$ and the resolvent formula 
\eqref{resolvent} extend to $Z_{\Lambda^{b}}:=\rho(\Delta_{\Lambda^{b}})\cap \CO\backslash(-\infty,0]$. The choice $$b_{11}=\frac1{b_{\-}-b_{\+}}\,,\quad b_{12}=\frac{{b_{\-}+b_{\+}}}{b_{\-}-b_{\+}}\,,\quad b_{22}=\frac{b_{\-}b_{\+}}{b_{\-}-b_{\+}}\,,$$ gives $\Delta_{\Lambda^{b}}=\Delta_{\Omega_{\-}}^{R}\oplus \Delta_{\Omega_{\+}}^{R}$, where $\Delta_{\Omega_{\-/\+}}^{R}$ denotes the Laplacian in $L^{2}(\Omega_{\-/\+})$ with Robin boundary conditions $\gamma_{1}^{\-/\+}u_{\-/\+}=b_{\-/\+}\gamma_{0}^{\-/\+}u_{\-/\+}$ (see \cite[Section 5.3]{JDE}). Notice that, since $\gamma_{1}^{\-/\+}$ are both defined in terms of the outward normal vector, the case describing the same Robin boundary conditions at both sides of $\Gamma$ corresponds to the choice $b_{\-}=b=-b_{\+}$ (thus $b_{11}=\frac1{2b}$, $b_{12}=0$, $b_{22}=-\frac{b}{2}$). 
\par
Arguing as in \cite[page 1480]{JST}, one shows that $\Delta_{\Lambda^{b}}$ is bounded from above; moreover $\ran(\Lambda_{z}^{b})=L^{2}(\Gamma)\oplus H^{1/2}(\Gamma)$ is compactly embedded in $\K^{*}=B^{-3/2}_{2,2}(\Gamma)\oplus H^{-1/2}(\Gamma)$. Thus Theorem \ref{scattering} applies and the limit operator $\Lambda^{b,+}_{\lambda}$ exists for any $\lambda\in E^{-}_{\Lambda^{b}}$ and $\Lambda^{b,+}_{\lambda}=(M_{\lambda}^{b,+})^{-1}$, where 
$$
M^{b,+}_{\lambda}=\left[\begin{matrix} b_{11}+ \gamma_{0}\SL^{+}_{\lambda}&b_{12}+ \gamma_{0}\DL^{+}_{\lambda}\\
b_{12}^{*}+ \gamma_{1}\SL^{+}_{\lambda}&b_{22}+ \gamma_{1}\DL^{+}_{\lambda}\end{matrix}\right]\,.
$$
Proceeding exactly in the same way as in the proof of Lemma \ref{rob}, one shows that $M^{b,+}_{\lambda}$ 
is the sum of a sign-definite operator plus a compact one. Therefore Theorem \ref{cor}  applies to $F^{\Lambda^{b}}_{\lambda}$, $\lambda\in E^{-}_{DN}\cap E^{-}_{\Lambda^{b}}$.

\section{Inverse Scattering for the Laplace operator with boundary conditions on non-closed Lipschitz surfaces.} 
We focus now on the case of boundary conditions assigned on a relatively open subset $\Sigma$ of the boundary $\Gamma$ of the domain $\Omega$. In this framework $\Delta_{\Lambda}$ provides models of obstacles supported on the non-closed interface $\Sigma$; our aim is to determine $\Sigma$ from the knowledge of the Scattering Matrix by implementing  the Factorization Method. An important difference with respect to the previous case appears: in fact the crucial coercivity hypothesis in Theorem \ref{cor} (by Lemma \ref{crit}, $M^{+}_{\lambda}$ there needs to be coercive) fails to hold in the spaces $\X^{s}_{\sharp}$, which are made of functions defined on the whole $\Gamma$ (see Notation \ref{not}).
To avoid such a problem one introduces (as in \cite{JDE} and \cite{JST}) projectors onto subspaces of functions supported on $\Sigma$. 
In the following, given $X\subset\Gamma$ closed, we use the definition $$H^{s}_{X}(\Gamma):=\{\phi\in H^{s}(\Gamma):\supp(\phi)\subseteq X\}\,.$$ Given $\Sigma\subset \Gamma$ relatively open with a Lipschitz boundary,  we denote by $\Pi_{\Sigma}$ the orthogonal projector in the Hilbert space $H^{s}(\Gamma)$, $|s|\le 1$, such that $\ran(\Pi_{\Sigma})$ is the subspace orthogonal to $H^{s}_{\Sigma^{c}}(\Gamma)$.
\begin{lemma}\label{proj} The orthogonal projection $\Pi_{\Sigma}$ identifies with the restriction map 
$$
R_{\Sigma}:H^{s}(\Gamma)\to H^{s}(\Sigma)\,,\qquad R_{\Sigma}\phi:=\phi|\Sigma
$$ 
and  its dual $\Pi^{*}_{\Sigma}$ identifies with the embedding 
$$
R^{*}_{\Sigma}:H^{-s}_{\overline\Sigma}(\Gamma)\to H^{-s}(\Gamma)\,,\qquad R^{*}_{\Sigma}\phi=\phi\,.
$$
\end{lemma}
\begin{proof} By \cite[page 77]{McL}, the map 
$$\U_{\Sigma}: \ran(\Pi_{\Sigma})\to H^{s}(\Sigma)\,,\quad
{\rm U}_{\Sigma}(\Pi_{\Sigma}\phi):=(\Pi_{\Sigma}\phi)|\Sigma=\phi|\Sigma$$ is an unitary isomorphism. Therefore we can regard $H^{s}(\Sigma)$ as a closed subspace of $H^{s}(\Gamma)$. Using the decomposition $\phi=(\uno-\Pi_{\Sigma})\phi\oplus\U_{\Sigma}^{-1}(\phi|\Sigma)$, 
the restriction operator $R_{\Sigma}\phi:=0\oplus{\rm U}_{\Sigma}\Pi_{\Sigma}\phi=0\oplus(\phi|\Sigma)$ is the orthogonal projection from $H^{s}(\Gamma)\simeq H^{s}_{\Sigma^{c}}(\Gamma)\oplus H^{s}(\Sigma)$ onto $H^{s}(\Sigma)$. Thus, using the identifications $\ran(\Pi_{\Sigma})\simeq H^{s}(\Sigma)$ and $H^{-s}_{\overline\Sigma}(\Gamma)\simeq H^{s}(\Sigma)^{*}$ (see, e.g., \cite[Lemma 4.3.1]{HW}), the orthogonal projection $\Pi_{\Sigma}$ identifies with $R_{\Sigma}$ 
and  its dual $\Pi^{*}_{\Sigma}$ identifies with $R^{*}_{\Sigma}$.
\end{proof}
\begin{remark}\label{RS} Let us notice that if a bounded linear operator $M:H^{-s}(\Gamma)\to H^{s}(\Gamma)$ is coercive then $R_{\Sigma}M R_{\Sigma}^{*}:H^{-s}_{\overline\Sigma}(\Gamma)\to H^{s}(\Sigma)$ is coercive as well by
$$
|\langle\phi,R_{\Sigma}M R_{\Sigma}^{*}\phi\rangle_{H^{-s}_{\overline\Sigma}(\Gamma),H^{s}(\Sigma)}|=|\langle R^{*}_{\Sigma}\phi,M R_{\Sigma}^{*}\phi\rangle_{H^{-s}(\Gamma),H^{s}(\Gamma)}|\ge c\,\|R_{\Sigma}^{*}\phi\|^{2}_{H^{-s}(\Gamma)}=c\,\|\phi\|^{2}_{H_{\overline\Sigma}^{-s}(\Gamma)}\,.
$$
Therefore (see Remark \ref{inv}) $(R_{\Sigma}M R_{\Sigma}^{*})^{-1}\in\B(H^{s}(\Sigma),H_{\overline\Sigma}^{-s}(\Gamma))$. Moreover, if $M=M_{\circ}+K$ with $M_{\circ}$ sign-definite and $K$ compact, then $R_{\Sigma}M R_{\Sigma}^{*}=R_{\Sigma}M_{\circ}R_{\Sigma}^{*}+R_{\Sigma}KR^{*}_{\Sigma}$, with $R_{\Sigma}M_{\circ}R_{\Sigma}^{*}$  sign-definite and $R_{\Sigma}KR^{*}_{\Sigma}$ compact. Analogously, if Im$\langle\phi,M\phi\rangle_{H^{-s}(\Gamma),H^{s}(\Gamma)}=0$ implies $\phi=0$, then Im$\langle\phi,R_{\Sigma}M R_{\Sigma}^{*}\phi\rangle_{H_{\overline\Sigma}^{-s}(\Gamma),H^{s}(\Sigma)}=0$ implies $R^{*}_{\Sigma}\phi=0$ and hence $\phi=0$.\par
The same considerations apply to $M:H^{-s}(\Gamma)\oplus H^{-t}(\Gamma)\to H^{s}(\Gamma)\oplus H^{t}(\Gamma)$ and  $(R_{\Sigma}\oplus R_{\Sigma})M (R_{\Sigma}^{*}\oplus R_{\Sigma}^{*}):H^{-s}_{\overline\Sigma}(\Gamma)\oplus H^{-t}_{\overline\Sigma}(\Gamma)\to H^{s}(\Sigma)\oplus H^{t}(\Sigma)$.  
\end{remark}
In the following $\Gamma_{\!\circ}$ is the Lipschitz boundary of an open bounded set $\Omega_{\circ}\subset\RE^{3}$ and $\Sigma_{\circ}\subset\Gamma_{\!\circ}$ is relatively open with Lipschitz boundary.
\begin{lemma}\label{sigma} Let $\Sigma\subset \Gamma$  and $\Sigma_{\circ}\subset \Gamma_{\!\circ}$ such that $\RE^{3}\backslash(\Sigma_{\circ}\cup\Sigma)$ is connected. Then 
$$
\Sigma_{\circ}\subset\Sigma\iff \phi^{\Sigma_{\circ}}_{\lambda}\in\ran(L^{\sharp}_{\lambda}|H^{s-s_{\sharp}}_{\overline \Sigma}(\Gamma))\,,\quad \sharp=D,N\,,
$$
where
\be\label{phiS}
\phi^{\Sigma_{\circ}}_{\lambda}(\xi):=\int_{\Sigma_{\circ}}\phi_{\lambda}^{x}(\xi)\,d\sigma_{\Gamma_{\!\circ}}(x)\equiv \int_{\Sigma_{\circ}}e^{i|\lambda|^{1/2}\xi\cdot x} \,d\sigma_{\Gamma_{\!\circ}}(x)
\,.
\ee
\end{lemma}
\begin{proof} Let $\widetilde u^{\sharp}_{\lambda,\phi}$ be the radiating (i.e satisfying the Sommerfeld radiating condition) solution in $\RE^{3}\backslash\overline\Sigma$ of Helmholtz equation $(-\Delta+\lambda)\widetilde u^{\sharp}_{\lambda,\phi}=0$ with either Dirichlet (whenever $\sharp=D$) or Neumann (whenever $\sharp=N$) boundary condition $\phi\in H^{s_{\sharp}}(\Sigma)$. Such a solution exists and is unique in 
\begin{align*}
&H^{1}_{\Delta,loc}(\RE^{3}\backslash\overline\Sigma):=\\
&\{u\in{\mathscr D}^{\prime}(\RE^{3}\backslash\overline\Sigma): u_{B}\in H^{1}(B\cap\RE^{3}\backslash\overline\Sigma),\ \Delta u_{B}\in L^{2}(B\cap\RE^{3}\backslash\overline\Sigma)\ \text{for any open ball $B\supset\overline\Omega$}\}\,,
\end{align*} 
where $u_{B}:=u|B\cap\RE^{3}\backslash\overline\Sigma$ (see \cite[Theorems 3.1 and 3.3]{Agra}, see also \cite[Section 12.8]{Agra-book} and, for the case with smooth boundaries, \cite{Steph}). Then (see, e.g., \cite[Exercise 9.4(iv)]{McL}) there exists a unique $\widetilde u^{\sharp,\infty}_{\lambda,\phi}\in\C^{\infty}({\mathbb S}^{2})$ such that 
$$ \widetilde  u^{\sharp}_{\lambda,\phi}(x)=\frac{e^{i\,|\lambda|^{1/2}\|x\|}}{4\pi\,\|x\|}\, \widetilde  u^{\sharp,\infty}_{\lambda,\phi}(\hat x)+O(\|x\|^{-2})\quad\text{as $\|x\|\to+\infty$, uniformly in $\hat x:=x/\|x\|$.}
$$
This defines the data-to-pattern operator 
$$\widetilde  K^{\sharp}_{\lambda}:H^{s_{\sharp}}(\Sigma)\to L^{2}({\mathbb S}^{2})\,,\quad \widetilde K^{\sharp}_{\lambda}\phi:=\widetilde u^{\sharp,\infty}_{\lambda,\phi}\,.
$$
Introducing the Herglotz operators $$
\widetilde H^{\sharp}_{\lambda}:L^{2}({\mathbb S}^{2})\to H^{s_{\sharp}}(\Sigma)\,,\qquad 
\widetilde H^{\sharp}_{\lambda}:=R_{\Sigma} H^{\sharp}_{\lambda}\,,
$$ 
where $H^{\sharp}_{\lambda}$ is defined in \eqref{Herg}, one has, for any $\phi\in H_{\overline \Sigma}^{-s_{\sharp}}(\Gamma)$ and $f\in L^{2}({\mathbb S}^{2})$,
$$
\langle L_{\lambda}^{\sharp}\phi,f\rangle_{L^{2}({\mathbb S}^{2})}=\frac1{2^{1/2}}\,\frac{|\lambda|^{1/4}}{(2\pi)^{3/2}}\, \langle\phi,\widetilde  H^{\sharp}_{\lambda}f\rangle_{H_{\overline \Sigma}^{-s_{\sharp}}(\Gamma), H^{s_{\sharp}}(\Sigma)}
=\frac1{2^{1/2}}\,\frac{|\lambda|^{1/4}}{(2\pi)^{3/2}}\, \langle (\widetilde {H}^{\sharp}_{\lambda})^{*}\phi, f\rangle_{L^{2}({\mathbb S}^{2})}\,.
$$
Proceeding as in \cite[proofs of Theorems 1.15 and 1.26]{KG} leading to \eqref{Herg1}, one gets
$$
(\widetilde H^{D}_{\lambda})^{*}=\widetilde K^{D}_{\lambda}R_{\Sigma}\gamma_{0}\SL^{+}_{\lambda}R_{\Sigma}^{*}\,,
\qquad (\widetilde H^{N}_{\lambda})^{*}=\widetilde K^{N}_{\lambda}R_{\Sigma}\gamma_{1}\DL^{+}_{\lambda}R_{\Sigma}^{*}\,,
$$
and so
$$
L^{D}_{\lambda}=\frac1{2^{1/2}}\,\frac{|\lambda|^{1/4}}{(2\pi)^{3/2}}\,\widetilde K^{D}_{\lambda}R_{\Sigma}\gamma_{0}\SL^{+}_{\lambda}R_{\Sigma}^{*}\,,\qquad 
L^{N}_{\lambda}=\frac1{2^{1/2}}\,\frac{|\lambda|^{1/4}}{(2\pi)^{3/2}}\,\widetilde K^{N}_{\lambda}R_{\Sigma}\gamma_{1}\DL^{+}_{\lambda}R_{\Sigma}^{*}\,.
$$
By the mapping properties of $\SL^{+}_{\lambda}$ and $\DL^{+}_{\lambda}$  and by Remark \ref{RS}, one has $R_{\Sigma}\gamma_{0}\SL^{+}_{\lambda}R_{\Sigma}^{*}\in\B(H_{\overline\Sigma}^{s-1/2}(\Gamma),H^{s+1/2}(\Sigma))$  
and $R_{\Sigma}\gamma_{1}\DL^{+}_{\lambda}R_{\Sigma}^{*}\in\B(H_{\overline\Sigma}^{s+1/2}(\Gamma),H^{s-1/2}(\Sigma))$, $s\in [0,1/2]$. These maps 
are bijections (by \eqref{exD}, \eqref{exN} in next Subsections \ref{Ds} and \ref{Ns} and by the regularity results in \cite[Theorem 3]{costa}; see also \cite{Steph} for the case of smooth boundaries), and so  
$$
\ran(L^{\sharp}_{\lambda}|H_{\overline\Sigma}^{s-s_{\sharp}}(\Gamma))=\ran(\widetilde K^{\sharp}_{\lambda}|H^{s+s_{\sharp}}(\Sigma))
\,.
$$ 
Therefore to conclude the proof we need to show that 
$$
\Sigma_{\circ}\subset\Sigma\iff \phi^{\Sigma_{\circ}}_{\lambda}\in
\ran(\widetilde K^{\sharp}_{\lambda}|H^{s+s_{\sharp}}(\Sigma))\,.
$$
Here we follows the same kind of reasonings as in \cite[Section 3.2]{KK}. Assume that $\Sigma_{\circ}\subset\Sigma$; let $u_{\lambda}^{\sharp,\Sigma_{\circ}}$ be defined according to 
$$
u_{\lambda}^{D,\Sigma_{\circ}}:=\SL_{\lambda}^{+}1_{\Sigma_{\circ}}\,,\qquad u_{\lambda}^{N,\Sigma_{\circ}}:=\DL_{\lambda}^{+}1_{\Sigma_{\circ}}\,.
$$
It solves the Helmoltz equation $(-\Delta+\lambda)u_{\lambda}^{\sharp,\Sigma_{\circ}}=0$ in $\RE^{3}\backslash\overline{\Sigma_{\circ}}$ and hence in $\RE^{3}\backslash\overline{\Sigma}$ as well. 
Let $\phi^{D}_{\Sigma_{\circ}}:=R_{\Sigma}\gamma_{0} u_{\lambda}^{D,\Sigma_{\circ}}\in H^{1/2}(\Sigma)$, $\phi^{N}_{\Sigma_{\circ}}:=R_{\Sigma}\gamma_{1} u_{\lambda}^{N,\Sigma_{\circ}}\in H^{-1/2}(\Sigma)$. Then $\widetilde K^{\sharp}_{\lambda}\phi^{\sharp}_{\Sigma_{\circ}}=\phi^{\Sigma_{\circ}}_{\lambda}$. Suppose now that $\Sigma_{\circ}\cap\Sigma^{c}\not=\emptyset$. Let $B\subset\RE^{3}$ be an open ball such that $\overline B\cap\overline\Sigma=\emptyset$, $B\cap\Sigma_{\circ}\not=\emptyset$. Assume that $\phi^{\Sigma_{\circ}}_{\lambda}=\widetilde K^{\sharp}_{\lambda}\phi_{\sharp}$ for some $\phi_{\sharp}\in H^{s+s_{\sharp}}(\Sigma)$ and consider the corresponding radiating solution $\widetilde u^{\sharp}_{\lambda,\phi_{\sharp}}$. Then, since $\widetilde K^{\sharp}_{\lambda}\phi_{\sharp}=\widetilde K^{\circ,\sharp}_{\lambda}\phi^{\circ,\sharp}_{\Sigma_{\circ}}$ (here the apex ${\!\,}^\circ$ denotes objects defined by using the surface $\Gamma_{\!\circ}$), one has, by Rellich's Lemma and unique continuation, $\widetilde u^{\sharp}_{\lambda,\phi_{\sharp}}|\RE^{3}\backslash(\Sigma_{\circ}\cup\Sigma)=u_{\lambda}^{\circ,\sharp,\Sigma_{\circ}}|\RE^{3}\backslash(\Sigma_{\circ}\cup\Sigma)$.  By elliptic regularity, 
$(-\Delta+\lambda)\widetilde u^{\sharp}_{\lambda,\phi}|B=0$ implies 
$\widetilde u^{\sharp}_{\lambda,\phi}|B\in H^{2}(B)$; this leads to a contradiction, since $u_{\lambda}^{\circ,\sharp,\Sigma_{\circ}}|B\notin H^{2}(B)$.
\end{proof}
By the same kind of proof provided for Corollary \ref{cLL}, one gets the following:
\begin{corollary}\label{cLLs} 
Let $\Sigma\subset \Gamma$  and $\Sigma_{\circ}\subset \Gamma_{\!\circ}$  such that $\RE^{3}\backslash(\Sigma_{\circ}\cup\Sigma)$ is connected. Then 
$$
\Sigma_{\circ}\subset\Sigma\iff \phi^{\Sigma_{\circ}}_{\lambda}\in\ran(L^{DN}_{\lambda}|H^{s-1/2}_{\overline \Sigma}(\Gamma)\oplus H^{t+1/2}_{\overline \Sigma}(\Gamma))\,,\quad s,t\in[0,1/2]\,.
$$
\end{corollary}
\begin{notation} We introduce the spaces
$$
\widetilde\X_{D}^{s}:=H^{1/2-s}(\Sigma)\,,\quad \widetilde\X_{N}^{s}:=H^{-1/2-s}(\Sigma) \,,\quad 
\widetilde\X_{DN}^{s}:=H^{1/2-s}(\Sigma)\oplus H^{-s}(\Sigma)\,,\quad 0\le s\le 1/2\,,
$$
so that 
$$
(\widetilde\X_{D}^{s})^{*}:=H^{s-1/2}_{\overline\Sigma}(\Gamma)\,,\quad (\widetilde\X_{N}^{s})^{*}:=H^{s+1/2}_{\overline\Sigma}(\Gamma) \,,\quad 
(\widetilde\X_{DN}^{s})^{*}:=H^{s-1/2}_{\overline\Sigma}(\Gamma)\oplus H^{s}_{\overline\Sigma}(\Gamma)\,,\quad 0\le s\le 1/2\,.
$$
\end{notation}
The following results is the analogue for screens of Theorem \ref{cor} :
\begin{theorem}\label{screen} Let $$F^{\Lambda}_{\lambda}=L^{\sharp}_{\lambda}\Lambda^{+}_{\lambda}{L_{\lambda}^{\sharp}}^{*}\,,\quad \lambda\in  E^{-}_{\Lambda}\,,\quad\sharp=D,N, DN\,,
$$ 
and suppose that  $\Lambda_{\lambda}^{+}=R^{*}_{\Sigma}(M^{+,\Sigma}_\lambda)^{-1}\!R_{\Sigma}$, where the bijection $M^{+,\Sigma}_\lambda\in\B((\widetilde\X_{\sharp}^{s})^{*},\widetilde\X_{\sharp}^{s})$, $s\in[0,1/2]$, has the decomposition $M^{+,\Sigma}_\lambda=M_{\circ}^{+,\Sigma}+K^{+,\Sigma}_{\lambda}$, where $M_{\circ}^{+,\Sigma}$ is sign-definite and $K^{+,\Sigma}_{\lambda}$ is compact. Let $\Sigma_{\circ}\subset \Gamma_{\!\circ}$ such that $\RE^{3}\backslash(\Sigma_{\circ}\cup\Sigma)$ is connected; then 
\begin{align}\label{formula}
\Sigma_{\circ}\subset\Sigma\iff\inf_{\substack{\psi\in L^{2}({\mathbb S}^{2})\\ \langle\psi,\phi^{\Sigma_{\circ}}_{\lambda}\rangle_{L^{2}({\mathbb S}^{2})}=1}}\left|\langle\psi,F^{\Lambda}_{\lambda}\psi\rangle_{L^{2}({\mathbb S}^{2})}\right|>0
\iff\sum_{k=1}^{\infty}\frac{|\langle\phi^{\Sigma_{\circ}}_{\lambda},\psi^{\Lambda}_{\lambda,k}\rangle_{L^{2}({\mathbb S}^{2})}|^{2}}{|z^{\Lambda}_{\lambda,k}|}<+\infty\,,
\end{align}
where the sequences $\{z^{\Lambda}_{\lambda,k}\}_{1}^{\infty}\subset\CO\backslash\{0\}$ and $\{\psi^{\Lambda}_{\lambda,k}\}_{1}^{\infty}\subset L^{2}({\mathbb S}^{2})$ provide the spectral resolution of $F^{\Lambda}_{\lambda}$ as in Remark \ref{rem-cn} and $\phi^{\Sigma_{\circ}}_{\lambda}$ is defined in \eqref{phiS}.
\end{theorem}
\begin{proof} We use the factorization $F^{\Lambda}_{\lambda}=\big(L^{\sharp}_{\lambda}R^{*}_{\Sigma}(M^{+,\Sigma}_\lambda)^{-1}\big)(M^{+,\Sigma}_{\lambda})^{*}\big(L^{\sharp}_{\lambda}R^{*}_{\Sigma}(M^{+,\Sigma}_\lambda)^{-1}\big)^{*}$. 
By proceeding as in the proof of Theorem \ref{cor}  (where now $B_{\lambda}=L^{\sharp}_{\lambda}R^{*}_{\Sigma}(M_{\lambda}^{+,\Sigma})^{-1}$), one gets $\text{\rm Im}\langle \phi, M^{+,\Sigma}_\lambda\phi\rangle_{(\widetilde\X_{\sharp}^{s})^{*},\widetilde\X_{\sharp}^{s}}$ $\not=0$ for any $\phi\not=0$. Since $M^{+,\Sigma}_\lambda$ is a bijection, one has 
$\ran(L^{\sharp}_{\lambda}R^{*}_{\Sigma}(M^{+,\Sigma}_\lambda)^{-1})=\ran(L^{\sharp}_{\lambda}|(\widetilde\X_{\sharp}^{s})^{*})$. Hence, by Lemma \ref{sigma}, Corollary \ref{cLLs} and by \cite[Theorem 1.16]{KG},  
$$
\Sigma_{\circ}\subset\Sigma\iff\inf_{\substack{\psi\in L^{2}({\mathbb S}^{2})\\ \langle\psi,\phi^{\Sigma_{\circ}}_{\lambda}\rangle_{L^{2}({\mathbb S}^{2})}=1}}\left|\langle\psi,F^{\Lambda}_{\lambda}\psi\rangle_{L^{2}({\mathbb S}^{2})}\right|>0\,.
$$
By proceeding as in the proof of Theorem \ref{sum}, $\ran(L^{\sharp}_{\lambda}R^{*}_{\Sigma}(M^{+,\Sigma}_\lambda)^{-1})=\ran(|\widetilde F^{\Lambda}_{\lambda}|^{1/2})$ and then one concludes by the same arguments.
\end{proof}
\begin{remark} If $M^{+,\Sigma}_{\lambda}$ in Theoren \ref{screen} is merely coercive, then the ``inf'' criterion still holds.
\end{remark}

\subsection{Applications}\label{appli-screens} Here we apply Theorem \ref{screen}  the analogue of models in the examples considered in Section \ref{appli} where now the boundary conditions holds only on $\Sigma$.  Before considering the specific examples, let us explain our strategy. \par
At first, notice that all the examples in Section \ref{appli} consider self-adjoint operators $\Delta_{\Lambda}$ with $\Lambda_{z}=M_{z}^{-1}$, where the map $z\mapsto M_{z}$ satisfies \eqref{M1-M2} (see Remark \ref{MM}). Hence, by Lemma \ref{Im-not-zero}, $\text{\rm Im}\langle \phi, M_{z}\phi\rangle_{{\X_{\sharp}^{s}}^{*},\X_{\sharp}^{s}}$ $\not=0$ for any $\phi\not=0$. Furthermore, all such $M_{z}$'s have a decomposition $M_{z}=M_{\circ}+K_{z}$ with $M_{\circ}$ sign-definite, $K_{z}$ compact; this can be checked by proceeding as in the proof of Lemma \ref{rob} using identities \eqref{Rlim}. Then, by Lemma \ref{crit}, $M_{z}$ is coercive. Now, the dual couple of projectors $R_{\Sigma}$, $R^{*}_{\Sigma}$ in Lemma \ref{proj} come into play: by Remark \ref{RS}, these properties of $M_{z}$ transfer to $M_{z}^{\Sigma}:=R_{\Sigma}M_{z}R^{*}_\Sigma$ (here and in the following lines, $R_{\Sigma}$ has to be replaced by $R_{\Sigma}\oplus R_{\Sigma}$ when one considers example in Subsection \ref{rob-ob}), and so, in particular, $M_{z}^{\Sigma}$ is coercive and $(R_{\Sigma}M_{z}R^{*}_\Sigma)^{-1}\in \B(\widetilde\X_{\sharp}^{s},(\widetilde\X_{\sharp}^{s})^{*})$. Then, setting 
\be\label{lambdatilde}
\widetilde\Lambda_{z}:=R^{*}_{\Sigma}(R_{\Sigma}M_{z}R^{*}_\Sigma)^{-1}R_{\Sigma}\,,
\ee
it is immediate to check that $z\mapsto \widetilde\Lambda_{z}\in \B(\X_{\sharp}^{s},{\X_{\sharp}^{s}}^{*})$ satisfies \eqref{Lambda}, and so, by Theorem \ref{WO}, it defines a self-adjoint operator $\Delta_{\widetilde\Lambda}$. Such an operator describes the model corresponding to the same kind of boundary conditions associated to $\Delta_{\Lambda}$, now assigned only on  $\Sigma$ (see \cite[Section 6]{JDE}, \cite[Section 7]{JST}). Since the limit operator $M^{+}_{\lambda}$ exists  (use \eqref{fLAP}) and,  by Theorem \ref{LH}, the limit $\widetilde\Lambda^{+}_{\lambda}$ exists as well, one gets $\widetilde\Lambda^{+}_{\lambda}=R^{*}_{\Sigma}(R_{\Sigma}M^{+}_{\lambda}R^{*}_\Sigma)^{-1}R_{\Sigma}$. Now, since all $M^{+}_{\lambda}$ appearing in the examples in Section \ref{appli} decompose as the sum of a sign-definite operator plus a compact one, by Remark \ref{RS} the same is true for   $R_{\Sigma}M^{+}_{\lambda}R^{*}_\Sigma$. 
In conclusion, the assumptions  in  Theorem \ref{screen} hold for any $\widetilde\Lambda$ defined as in \eqref{lambdatilde} where $M_{z}$ is any of the operators given in the examples in Section \ref{appli}; hence the reconstruction formula \eqref{formula} applies to $F_{\lambda}^{\widetilde\Lambda}$. In what follows this scheme is implemented case by case.
\subsubsection{Dirichlet screens.}\label{Ds} One considers $\Delta_{\widetilde\Lambda^{D}}$ with $$\widetilde\Lambda^{D}_{z}=-R_{\Sigma}^{*}( R_{\Sigma}\gamma_{0}\SL_{z}R^{*}_{\Sigma})^{-1}R_{\Sigma}\in\B(H^{1/2}(\Gamma),H^{-1/2}(\Gamma))\,,\qquad z\in\CO\backslash\RE\,.$$ $\Delta_{\widetilde\Lambda^{D}}$ is a (bounded from above) self-adjoint representation of the Laplacian on $\RE^{3}\backslash\overline\Sigma$ with homogeneous Dirichlet boundary conditions at $\Sigma$ (see \cite[Example 7.1]{JST}). 
By \cite[Theorem 2.19]{CFP}, the map $z\mapsto \widetilde\Lambda^{D}_{z}$ and the corresponding resolvent formula \eqref{resolvent} extends to $ Z_{\widetilde\Lambda^{D}}:=\rho(\Delta_{\widetilde\Lambda^{D}})\cap \CO\backslash(-\infty,0]=\CO\backslash(-\infty,0]$. By \cite[Theorem 3.7]{JST}, $\sigma_{\rm p}^{-}(\Delta_{\widetilde\Lambda^{D}})$ is empty and so, by Theorem \ref{LH}, 
\be\label{exD}
\forall\lambda\in (-\infty,0)\,,\qquad( R_{\Sigma}\gamma_{0}\SL_{\lambda}^{+}R^{*}_{\Sigma})^{-1}\in\B(H^{1/2}(\Sigma),H^{-1/2}_{\overline\Sigma}(\Gamma))\,.
\ee  
Therefore Theorem  \ref{screen} applies to $
F^{\widetilde\Lambda^{D}}_{\lambda}$, $\lambda\in(-\infty,0)$. 
\subsubsection{Neumann screens.}\label{Ns} One considers $\Delta_{\widetilde\Lambda^{N}}$ with $$\widetilde\Lambda^{N}_{z}=-R_{\Sigma}^{*}( R_{\Sigma}\gamma_{1}\DL_{z}R^{*}_{\Sigma})^{-1}R_{\Sigma}\in\B(H^{-1/2}(\Gamma),H^{1/2}(\Gamma))\,,\qquad z\in\CO\backslash\RE\,.$$ 
$\Delta_{\widetilde\Lambda^{N}}$ is a (bounded from above) self-adjoint representation of the Laplacian on $\RE^{3}\backslash\overline\Sigma$ with homogeneous Neumann boundary conditions at $\Sigma$ (see \cite[Example 7.2]{JST}). By \cite[Theorem 2.19]{CFP}, the map $z\mapsto \widetilde\Lambda^{N}_{z}$ and the corresponding resolvent formula \eqref{resolvent} extends to $ Z_{\widetilde\Lambda^{N}}:=\rho(\Delta_{\widetilde\Lambda^{N}})\cap \CO\backslash(-\infty,0]=\CO\backslash(-\infty,0]$. By \cite[Theorem 3.7]{JST}, 
$\sigma_{\rm p}^{-}(\Delta_{\widetilde\Lambda^{N}})$ is empty and so, by Theorem \ref{LH}, 
\be\label{exN}
\forall\lambda\in (-\infty,0)\,,\qquad( R_{\Sigma}\gamma_{1}\DL_{\lambda}^{+}R^{*}_{\Sigma})^{-1}\in\B(H^{-1/2}(\Sigma),H^{1/2}_{\overline\Sigma}(\Gamma))\,.
\ee  
Therefore Theorem  \ref{screen} applies to $F^{\widetilde\Lambda^{N}}_{\lambda}$, $\lambda\in(-\infty,0)$. 
\subsubsection{Screens with semitransparent boundary conditions $\alpha_{\Sigma}\gamma_{0}u=[\gamma_{1}]u$.}\label{delta-s} Let $\alpha\in L^{\infty}(\Gamma)$ real-valued such that $\sgn(\alpha)$ is constant and $\frac1{\alpha}\in L^{\infty}(\Gamma)$; let us define $\alpha_{\Sigma}:=\alpha|\Sigma$ and $\alpha_{\Sigma}^{-1}\in \B(L^{2}_{\overline\Sigma}(\Gamma),L^{2}(\Sigma))$
 by $\alpha^{-1}_{\Sigma}\phi:=(\alpha^{-1}\phi)|\Sigma$. Since $-\left(\alpha_{\Sigma}^{-1}+ R_{\Sigma}\gamma_{0}\SL_{z}R^{*}_{\Sigma}\right)=R_{\Sigma}M^{\alpha}_{z}R^{*}_{\Sigma}$, where $M^{\alpha}_{z}$ is defined in \eqref{Malpha}, one considers $\Delta_{\widetilde \Lambda^{\alpha}}$, where 
$$\widetilde \Lambda_{z}^{\alpha}=R^{*}_{\Sigma}(R_{\Sigma}M^{\alpha}_{z}R^{*}_{\Sigma})^{-1}R_{\Sigma}\in\B(L^{2}(\Gamma))\,,\qquad z\in\CO\backslash(-\infty,0]\,.
$$
$\Delta_{\widetilde\Lambda^{\alpha}}$ is a self-adjoint representation of the (bounded from above) Laplacian on $\RE^{3}\backslash\overline\Sigma$ with boundary conditions at $\Sigma$ given by $\alpha_{\Sigma}R_{\Sigma}\gamma_{0}u=R_{\Sigma}[\gamma_{1}]u$, $R_{\Sigma}[\gamma_{0}]u=0$  (see \cite[Example 7.4]{JST}). By \cite[Theorem 2.19]{CFP}, the map $z\mapsto\widetilde\Lambda^{\alpha}_{z}$ and the resolvent formula \eqref{resolvent} extend to $Z_{\widetilde\Lambda^{\alpha}}:=\rho(\Delta_{\widetilde\Lambda^{\alpha}})\cap \CO\backslash(-\infty,0]$.  By \cite[Theorem 3.7]{JST},  $\sigma_{\rm p}^{-}(\Delta_{\widetilde \Lambda^{\alpha}})$ is empty and so, by Theorem \ref{LH},
$$\forall\lambda\in(-\infty,0)\,,\qquad \left(\alpha_{\Sigma}^{-1}+ R_{\Sigma}\gamma_{0}\SL_{\lambda}^{+}R^{*}_{\Sigma}\right)^{-1}\in\B(L^{2}(\Sigma),L^{2}_{\overline\Sigma}(\Gamma))\,.
$$  
Therefore Theorem \ref{screen} applies to $
F^{\widetilde\Lambda^{\alpha}}_{\lambda}$, $\lambda\in(-\infty,0)$.
\subsubsection{Screens with semitransparent boundary conditions ${\theta_{\Sigma}}\gamma_{1}u=[\gamma_{0}]u$.}\label{delta'-s} Let ${\theta}\in L^{p}(\Gamma)$, $p>2$; set $\theta_{\Sigma}:=\theta|\Sigma$ and define the corresponding operator $\theta_{\Sigma}\in\B(H_{\overline\Sigma}^{1/2}(\Gamma),H^{-1/2}(\Sigma))$ by $\theta_{\Sigma}\varphi:=(\theta\varphi)|\Sigma$. Since $\left({\theta_{\Sigma}}- R_{\Sigma}\gamma_{1}\DL_{z}R^{*}_{\Sigma}\right)=R_{\Sigma}M^{\theta}_{z}R^{*}_{\Sigma}$, where $M^{\theta}_{z}$ is defined in \eqref{Mtheta}, one considers $\Delta_{\widetilde \Lambda^{\theta}}$, where $$\widetilde \Lambda_{z}^{{\theta}}=R^{*}_{\Sigma}(R_{\Sigma}M_{z}^{\theta}R^{*}_{\Sigma})^{-1}R_{\Sigma}\in\B(H^{-1/2}(\Gamma),H^{1/2}(\Gamma))\,,\qquad z\in\CO\backslash\RE\,.
$$ 
$\Delta_{\widetilde\Lambda^{{\theta}}}$ is a self-adjoint representation of the (bounded from above) Laplacian on $\RE^{3}\backslash\overline\Sigma$ with boundary conditions  at $\Sigma$ given by ${\theta_{\Sigma}}R_{\Sigma}\gamma_{1}u=R_{\Sigma}[\gamma_{0}]u$, $R_{\Sigma}[\gamma_{1}]u=0$ (see \cite[Example 7.5]{JST}). By \cite[Theorem 2.19]{CFP}, the map $z\mapsto\widetilde\Lambda^{\theta}_{z}$ and the resolvent formula \eqref{resolvent} extend to $Z_{\widetilde\Lambda^{\theta}}:=\rho(\Delta_{\widetilde\Lambda^{\theta}})\cap \CO\backslash(-\infty,0]$. By \cite[Theorem 3.7]{JST}, $\sigma_{\rm p}^{-}(\Delta_{\widetilde \Lambda^{{\theta}}})$ is empty and so, by Theorem \ref{LH},   
$$
\forall\lambda\in(-\infty,0)\,,\qquad(\theta_{\Sigma}+ R_{\Sigma}\gamma_{1}\DL_{\lambda}^{+}R^{*}_{\Sigma})^{-1}\in\B(H^{-1/2}(\Sigma), H_{\overline\Sigma}^{1/2}(\Gamma))\,.
$$
Therefore Theorem \ref{screen} applies to $
F^{\widetilde\Lambda^{{\theta}}}_{\lambda}$, $\lambda\in(-\infty,0)$.
\subsubsection{Screens with local boundary  conditions.}\label{rob-s} 
Let $b_{11}\in L^{\infty}(\Gamma)$, $b^{-1}_{11}\in L^{\infty}(\Gamma)$, $b_{22}\in L^{p}(\Gamma)$, $p>2$, $b_{12}\in\C^{\kappa}(\Gamma)$, $0<\kappa<1$, with $b_{11}$ and $b_{22}$ real-valued, $b_{11}<0$. Set $b_{ij}^{\Sigma}:=b_{ij}|\Sigma$ and define the corresponding multiplication operator by $b_{ij}^{\Sigma}\phi:=(b_{ij}\phi)|\Sigma$, where $\supp(\phi)\subseteq\overline\Sigma$.
Since 
$$\left[\begin{matrix} b^{\Sigma}_{11}+ R_{\Sigma}\gamma_{0}\SL_{z}R_{\Sigma}^{*}&b^{\Sigma}_{12}+ R_{\Sigma}\gamma_{0}\DL_{z}R_{\Sigma}^{*}\\
(b^{\Sigma}_{12})^{*}+ R_{\Sigma}\gamma_{1}\SL_{z}R_{\Sigma}^{*}&b^{\Sigma}_{22}+ R_{\Sigma}\gamma_{1}\DL_{z}R_{\Sigma}^{*}\end{matrix}\right]
=(R_{\Sigma}\oplus R_{\Sigma})M^{b}_{z}(R^{*}_{\Sigma}\oplus R_{\Sigma}^{*})\,,
$$ 
where $M^{b}_{z}$ is defined in Lemma \eqref{rob}, one considers $\Delta_{\widetilde \Lambda^{b}}\,,
$ 
where 
$$
\widetilde \Lambda_{z}^{b}=(R^{*}_{\Sigma}\oplus R^{*}_{\Sigma})(M_{z}^{b})^{-1}(R_{\Sigma}\oplus R_{\Sigma})\in \B(L^{2}(\Gamma)\oplus H^{-1/2}(\Gamma), L^{2}(\Gamma)\oplus H^{1/2}(\Gamma))\,,\qquad z\in\CO\backslash\RE\,.
$$ 
$\Delta_{\widetilde\Lambda^{b}}$ is a self-adjoint representation of the (bounded from above, this follows proceeding as in \cite[page 1480]{JST}) Laplacian on $\RE^{3}\backslash\overline\Sigma$ with boundary conditions  at $\Sigma$ given by 
$$
\begin{cases}
R_{\Sigma}\gamma_{0}u=b^{\Sigma}_{11}R_{\Sigma}[\gamma_{0}]u+b^{\Sigma}_{12}R_{\Sigma}[\gamma_{1}]u\,,\quad\\
R_{\Sigma}\gamma_{1}u=(b^{\Sigma}_{12})^{*}R_{\Sigma}[\gamma_{0}]u+b^{\Sigma}_{22}R_{\Sigma}[\gamma_{1}]u\,.
\end{cases}
$$ 
By \cite[Theorem 2.19]{CFP}, the map $z\mapsto \widetilde \Lambda_{z}^{b}$ and the resolvent formula \eqref{resolvent} extend to $Z_{\widetilde\Lambda^{b}}:=\rho(\Delta_{\widetilde\Lambda^{b}})\cap \CO\backslash(-\infty,0]$. Since $\sigma_{\rm p}^{-}(\Delta_{\widetilde \Lambda^{b}})$ is empty (see \cite[Theorem 3.7]{JST}), 
$$
\left[\begin{matrix} b^{\Sigma}_{11}+ R_{\Sigma}\gamma_{0}\SL^{+}_{\lambda}R_{\Sigma}^{*}&b^{\Sigma}_{12}+ R_{\Sigma}\gamma_{0}\DL^{+}_{\lambda}R_{\Sigma}^{*}\\
(b^{\Sigma}_{12})^{*}+ R_{\Sigma}\gamma_{1}\SL^{+}_{\lambda}R_{\Sigma}^{*}&b^{\Sigma}_{22}+ R_{\Sigma}\gamma_{1}\DL^{+}_{\lambda}R_{\Sigma}^{*}\end{matrix}\right]
^{-1}\in\B(L^{2}(\Sigma)\oplus H^{-1/2}(\Sigma), L^{2}_{\overline\Sigma}(\Gamma)\oplus H_{\overline\Sigma}^{1/2}(\Gamma))$$ exists  for any $\lambda\in(-\infty,0)$ by Theorem \ref{LH}. Therefore Theorem \ref{screen} applies to $
F^{\widetilde\Lambda^{b}}_{\lambda}$, $\lambda\in(-\infty,0)$.



\begin{thebibliography}{99}
\bibitem{Agra} M. S. Agranovich: Strongly elliptic second order systems with spectral parameter in transmission conditions on a nonclosed surface. In: P. Boggiatto, L. Rodino, J. Toft, M. W. Wang (eds.), {\it Pseudo-differential operators and related topics}, Operator Theory:
Advances and Applications, vol. 164, 1-21, Birkh\"auser, Basel, 2006. 
\bibitem{Agra-book} M. S. Agranovich: {\it Sobolev Spaces, Their Generalizations and Elliptic Problems in Smooth and Lipschitz Domains.} Springer, Berlin, 2015.
\bibitem{BW} H. Baumg\"artel, M. Wollenberg: {\it Mathematical Scattering Theory}, Akademie-Verlag, Berlin, 1983.
\bibitem{BLL} J. Behrndt, M. Langer, V. Lotoreichik: Schr\"odinger 
operators with $\delta$- and $\delta'$-potentials supported on hypersurfaces. 
{\it Ann. Henri Poincar\'e} {\bf 14} (2013),  385-423.
\bibitem{BMN} J. Behrndt, M. Malamud, H. Neidhardt: Scattering matrices and Dirichlet-to-Neumann maps.  {\it J. Funct. Anal.} {\bf 273} (2017), 1970-2025.
\bibitem{BL13} O. Bondarenko, X. Liu: The factorization method for inverse obstacle scattering with conductive boundary condition. {\it Inverse Problems} {\bf 29} (2013), 095021 (25pp).
\bibitem{BH13} Y. Boukari, H. Haddar:
The factorization method applied to cracks with impedance boundary conditions. {\it  
Inverse Probl. Imaging} {\bf 7} (2013), 1123-1138. 
\bibitem {BEKS} J.F. Brasche, P. Exner, Y.A. Kuperin, P. \v Seba: Schr\"odinger operators with singular interactions. {\it J. Math. Anal. Appl.} {\bf 184}  (1994), 112-139.
\bibitem{CFP} C. Cacciapuoti, D. Fermi, A. Posilicano: On inverses of Kre\u\i n's $\mathscr Q$-functions, {\it Rend. Mat. Appl.} {\bf 39} (2018), 229-240.
\bibitem{costa} M. Costabel: Boundary Integral Operators on Lipschitz Domains: Elementary Results. { SIAM J. Math. Anal.} {19} (1988), 613-626. 
\bibitem{Grinb06} N. Grinberg: The Operator Factorization Method in Inverse Obstacle Scattering. 
{\it Integr. equ. oper. theory} {\bf 54} (2006), 333-348.
\bibitem{HW} G. C. Hsiao, W. L. Wendland: {\it Boundary Integral Equations.} Springer, Berlin, 2008.
\bibitem{JW} A. Jonsson, H. Wallin: Function Spaces on Subsets of $\RE^n$.
{\it Math. Reports}, {\bf 2} (1984), 1-221.
\bibitem{Jorg} K. J\"orgens: {\it Linear Integral Operators.} Pitman, London, 1982.
\bibitem{Kato2} T. Kato: Scattering theory with two Hilbert spaces. {\it J. Funct. Anal.} {\bf 1} (1967), 342-369.
\bibitem{Kato} T. Kato: {\it Perturbation Theory for Linear Operators.} Springer, Berlin, 1976
\bibitem{Kirsch98} A. Kirsch: Characterization of the shape of a scattering obstacle using the spectral data of the far field operator. {\it Inverse Problems} {\bf 14} (1998), 1489-1512.
\bibitem{KG} A. Kirsch, N. Grinberg: {\it The Factorization Method for Inverse Problems.} Oxford University Press, Oxford, 2008.
\bibitem{KK} A. Kirsch, A. Kleefeld: The factorization method for a conductive boundary condition. 
{\it J. Integral Equations Appl.} {\bf 24} (2012), 575-601. 
\bibitem{Kirsch00} A. Kirsch, S. Ritter: A linear sampling method for inverse scattering from an open arc. {\it Inverse Problems} {\bf 16} (2000), 89-105.
\bibitem{KomKop}  A. Komech, E. Kopylova: {\it Dispersion decay and scattering theory.} John Wiley \& Sons, 
Hoboken, New Jersey, 2012. 
\bibitem{JMPA} A. Mantile, A. Posilicano: Asymptotic Completeness and S-Matrix for Singular Perturbations, to appear in {\it J. Math. Pures Appl. } DOI: https://doi.org/10.1016/j.matpur 2019.01.017, 2019. 
\bibitem{JDE} A. Mantile, A. Posilicano, M. Sini: Self-adjoint elliptic operators with boundary conditions on not closed surfaces. {\it J. Differential Equations} {\bf 261} (2016), 1-55.
\bibitem{JST} A. Mantile, A. Posilicano, M. Sini: Limiting Absorption Principle, Generalized Eigenfunctions and Scattering Matrix for Laplace Operators with Boundary conditions on Hypersurfaces. {\it J. Spectr. Theory} {\bf 8} (2018), 1443-1486.
\bibitem{JDE2} A. Mantile, A. Posilicano, M. Sini: Uniqueness in inverse acoustic scattering with unbounded gradient across Lipschitz surfaces. 
{\it J. Differential Equations} {\bf 265} (2018), 4101-4132.
\bibitem{McL} W. McLean: \emph{Strongly elliptic systems and boundary
\ integral equations.} Cambridge University Press, Cambrdige, 2000.
\bibitem {P01}A. Posilicano: A Kre\u{\i}n-like formula for singular
perturbations of self-adjoint operators and applications. \emph{J. Funct.
Anal.} \textbf{183} (2001), 109-147.
\bibitem {P04}A. Posilicano: Boundary triples and Weyl functions for
singular perturbations of self-adjoint operators. \emph{Methods Funct. Anal.
Topology}, \textbf{10} (2004), 57-63.
\bibitem {P08}A. Posilicano. Self-adjoint extensions of restrictions,
\emph{Oper. Matrices}, \textbf{2} (2008), 483-506.
\bibitem{Sch-book}  M. Schechter: {\it Operator Methods in Quantum Mechanics.} North Holland, New York, 1981 (Dover reprint, 2002).
\bibitem{RS} M. Reed, B. Simon: {\it Methods of Modern Mathematical Physics. Vol. III Scattering Theory.} Academic Press. London 1979.
\bibitem{Steph} E. P. Stephan: Boundary integral equations for screen problems in $\RE^{3}$. {\it Integral Equations Operator Theory}. {\bf 10} (1987), 236-257. 
\bibitem{Trib} H. Triebel: {\it Fractals and Spectra.} Birkh\"auser. Basel, 1997.
\bibitem{Y} D. R. Yafaev. {\it Mathematical Scattering Theory. 
General theory}. American Mathematical Society, 1992.
\bibitem{W} C. H. Wilcox: {\it Scattering theory for the d'Alembert equation in exterior domains.} 
Lecture Notes in Mathematics, Vol. 442. Springer, Berlin 1975.
\end{thebibliography}
\end{document}